\documentclass[11pt]{amsart}
\usepackage[utf8]{inputenc}
\usepackage{amsfonts, amssymb, amsmath, amsthm, color, float,enumerate}
\usepackage{url}
\usepackage{bm}
\usepackage[unicode,psdextra]{hyperref}
\usepackage{pgfplots,tikz}
\pgfplotsset{compat=1.18}

\title[Mixed weak inequalities of Fefferman-Stein type for commutators of CZO]{Revisiting mixed weak inequalities of Fefferman-Stein type for commutators of Calderón-Zygmund operators: an improvement}
\author{}
\date{}

\usepackage[a4paper, left=2cm, right=2cm, top=3cm, bottom=3cm]{geometry} 

\theoremstyle{plain}
   \newtheorem{teo}{Theorem}
   
   \newtheorem{lema}[teo]{Lemma}
   \newtheorem{propo}[teo]{Proposition}
   
\theoremstyle{definition}
   
\theoremstyle{remark}
 \newtheorem{obs}{Remark}

\numberwithin{equation}{section}
\numberwithin{teo}{section}
\allowdisplaybreaks

\definecolor{aquamarine}{rgb}{0.5, 1.0, 0.83}
\definecolor{americanrose}{rgb}{1.0, 0.01, 0.24}
\definecolor{arsenic}{rgb}{0.23, 0.27, 0.29}
\definecolor{blizzardblue}{rgb}{0.67, 0.9, 0.93}
\definecolor{blush}{rgb}{0.87, 0.36, 0.51}
\definecolor{celestialblue}{rgb}{0.29, 0.59, 0.82}
\definecolor{chocolate(web)}{rgb}{0.82, 0.41, 0.12}
\definecolor{brightpink}{rgb}{1,0,0.5}
\definecolor{cadmiunred}{rgb}{0.89,0,0.13}

\hypersetup{
	colorlinks = true,
	linkcolor = blue,
	anchorcolor = blue,
	citecolor = blue,
	filecolor = blue,
	urlcolor = blue
}

\begin{document}

\author[R. Ayala]{Rocío Ayala}
	\address{Rocío Ayala, CONICET and Departamento de Matem\'{a}tica (FIQ-UNL), Santa Fe, Argentina.}
	\email{rocioayalazara@gmail.com}
	
\author[F. Berra]{Fabio Berra}
\address{Fabio Berra, CONICET and Departamento de Matem\'{a}tica (FIQ-UNL), Santa Fe, Argentina.}
\email{fberra@santafe-conicet.gov.ar}

\author[G. Pradolini]{Gladis Pradolini}
\address{Gladis Pradolini, CONICET and Departamento de Matem\'{a}tica (FIQ-UNL), Santa Fe, Argentina.}
\email{gladis.pradolini@gmail.com}

\thanks{The authors were supported by Consejo Nacional de Investigaciones Científicas y Técnicas (CONICET), Universidad Nacional del Litoral and Gobierno de la Provincia de Santa Fe (Argentina)}



\begin{abstract}	
In this paper we establish mixed weak inequalities of Fefferman-Stein type for Calderón-Zygmund operators and their commutators, improving some previous results known in the literature. The main estimates also generalize the classical weighted weak Fefferman-Stein inequalities proved in \cite{P94} and \cite{PP01}. In order to obtain the main results, our approach is to give a strong Fefferman-Stein type inequality for the operators involved with respect to an adequate measure.
\end{abstract}
\maketitle

\section{Introduction and main results}\label{section 1}
In \cite{S85} E. Sawyer, searching  for an alternative proof for the boundedness of the Hardy-Littlewood maximal operator $M$ between weighted Lebesgue spaces with $A_p$ weights, proved that the following mixed weak type inequality
\begin{equation}\label{Sawyer og}
    uv\left( \left\{ x \in \mathbb{R}: \frac{M(fv) (x)}{v(x)} > \lambda \right\} \right) \leq \frac{C}{\lambda} \int_{\mathbb{R}} |f(x)| u(x) v(x) \,dx
\end{equation} holds for $u,v$
 in the $A_1$ Muckenhoupt class and for every positive $\lambda$. The inequality above establishes that the operator $S_{v}f= M(fv) v^{-1}$ is of weak $(1,1)$ type with respect to the measure $uv$. In the same article, the author also conjectured that a similar estimate can be achieved for the Hilbert transform. 

Subsequently, in \cite{CUMPconjecture}, the authors not only extended Sawyer's result to higher dimensions  but also they solved Sawyer's conjecture for Calderón-Zygmund operators (CZO). Moreover they considered an alternative result by using different hypothesis on the weights involved. Some extensions of these results were also given in \cite{LOP19}.

Later on, a corresponding result in the spirit of inequality \eqref{Sawyer og} for higher order commutators of CZO with BMO symbols was obtained in  \cite{BCP19a}. Similar estimations for fractional integral operators and their commutators were proved in \cite{BCP19}.

On the other hand, in \cite{P94} C. Pérez proved the following weak inequality of Fefferman-Stein type
\begin{equation}\label{P1994}
    u( \left\{x \in \mathbb{R}^{n} : |Tf(x)| > \lambda \right\}) \leq \frac{C}{\lambda} \int_{\mathbb{R}^n}|f(x)| M_{\Phi_{\varepsilon}}u(x) \,dx,
\end{equation} 
 where $T$ is a CZO, $u$ is a weight and $M_{\Phi_\varepsilon}$ is a generalized maximal operator associated to the Young function $\Phi_{\varepsilon} (t) =t (1+\log^{+}t)^\varepsilon$, $\varepsilon >0$. 
 Furthermore, in \cite{PP01} the authors extended the result in \cite{P94} for higher order commutators of $T$ with BMO symbols. Concretely, they proved that 
 \begin{equation*}
      u( \left\{x \in \mathbb{R}^{n} : |T_{b}^{m}f(x)| > \lambda \right\}) \leq C \int_{\mathbb{R}^n} \Phi_{m} \left( \|b \|_{\text{BMO}}^{m}\frac{|f(x)|}{\lambda} \right)M_{\Phi_{m+\varepsilon}}u(x) \,dx
 \end{equation*} for every positive $\lambda.$

Related to Sawyer estimates of Fefferman-Stein type of CZO, if $  v \in RH_{\infty}\cap A_q$ and $u$ is a weight, it was shown in \cite{BCP22} that the following mixed inequality 
\begin{equation}\label{BCP que mejoramos}
    uv\left( \left\{ x \in \mathbb{R}^n : \frac{|T(fv)(x)|}{v(x)} > \lambda \right\} \right) \leq \frac{C}{\lambda} \int_{\mathbb{R}^n} |f(x)| M_{\Phi_{\delta}, v^{1-q'}}u(x) M(\Psi(v))(x) \,dx
\end{equation}
holds for every positive $\lambda$, where $\Psi (t) = t^{p'+1-q'}\chi_{[0,1]}(t) + t^{p'}\chi_{[1,\infty)}(t)$ has certain restrictions on $\delta$, $p$ and $q$ 
(see below for the definition of the maximal operators appearing in \eqref{P1994} and \eqref{BCP que mejoramos}). Under similar hypotheses on the parameters and the weights involved, an extension of inequality \eqref{BCP que mejoramos} for higher order commutators of $T$ with BMO symbols was obtained in \cite{BPR25}. In both cases, when $v=1$ these results give those obtained in \cite{P94} and \cite{PP01}.

Our main goal in the present article is to give an improvement of the results established in \cite{BCP22} and \cite{BPR25} for CZO and their commutators,  generalizing also the inequalities in \cite{P94} and \cite{PP01}. Our approach is to obtain a version of a strong Fefferman-Stein type inequality for $T$ proved by C. Pérez in \cite{P94} with respect to a certain weight, as well as their higher order commutators.
\vspace{7mm}

We now introduce the definitions and notation required in order to state our main results. 

A function $\Phi:[0,\infty) \to [0,\infty)$ is called a Young function if there exists a non-zero, non-negative and increasing function $\phi$ that verifies $\lim_{s\to \infty} \phi(s) = \infty$ and $ \displaystyle \Phi(t) = \int_{0}^{t} \phi(u) \,du$. Then $\Phi$ is a continuous, convex and increasing function satisfying $\Phi(0)=0$ and $\displaystyle\lim_{t \rightarrow{\infty}} \Phi(t)/t = \infty$. We also consider $\Phi(t) = t$ as a Young function.

Given a Young function $\Phi$ and a weight $w$, we define the generalized weighted maximal operator associated to $\Phi$ by
     \begin{equation*}
         M_{\Phi, w}f(x) = \sup_{Q \ni x} \| f \|_{\Phi, Q, w},
     \end{equation*} where the supremum is taken over every cube $Q$ containing $x$ and 
     \begin{equation*}
         \|f \|_{\Phi,Q,w}= \inf \left\{ \lambda >0 : \frac{1}{w(Q)} \int_{Q} \Phi\left(\frac{|f(x)|}{\lambda}\right)w(x) \,dx \leq 1 \right\}.
     \end{equation*}
     This quantity is usually called the Luxemburg average of $f$ on the cube $Q$ with respect to the measure $d\mu(x)=w(x)dx$. When $w=1$ we shall denote $M_{\Phi,w}= M_{\Phi}$, the classical generalized maximal operator associated to the Young function $\Phi$. When $\Phi(t)=t$, $M_{\Phi, w}=M_{w}$ is the Hardy-Littlewood maximal operator associated to the measure $d\mu(x) = w(x)dx$. When $\Phi(t)=t^{r}$, with $r>1$, then $M_{\Phi,w}f= M_{r,w}f=M_{w}(|f|^{r})^{1/r}$. If  $k \in \mathbb{N}$, we shall also consider the iterated maximal operator $M^{k}_{w}$, the composition of $M_{w}$ with itself $k-$times. 

Throughout this paper we shall consider the Young function $\Phi_{\varepsilon}(t)= t (1+\log^{+} t)^{\varepsilon}$, $\varepsilon \geq 0$.  
\begin{obs}\label{obs maximal phi_m e iterada m+1}
    If $\Phi_{k}(t)= t (1+ \log^{+}t)^{k}$ with $k \in \mathbb{N}$, then it is well known that $M^{k+1}f \approx M_{\Phi_{k}}f$. (See for example \cite{BHP06}).
\end{obs}

We shall be dealing with a linear operator $T$, bounded on $L^2$ that verifies
\begin{equation*}
    Tf(x) =\int_{\mathbb{R}^{n}} K(x-y) f(y) \,dy 
\end{equation*}
for $f \in C_{0}^{\infty}$ and $x \notin \text{supp}f$.  
The kernel $K:\mathbb{R}^{n}\backslash \{\textbf{0}\} \to \mathbb{C}$ is a measurable function that satisfies the size condition
\begin{equation}\label{OCZ: condicion de tamanio}
    |K(x)| \leq \frac{C}{|x|^{n}} \hspace{7mm}\text{ for }x \in \mathbb{R}^n    \backslash\{\textbf{0}\}
\end{equation} 
and the smoothness condition
\begin{equation}\label{OCZ: condicion de suavidad}
    |K(x-y) - K(x-z)| \leq C \frac{|x-z|}{|x-y|^{n+1}} \hspace{5mm} \text{if } |x-y| > 2 |y-z|.
\end{equation} We shall say that $K$ is a standard kernel if both conditions above are satisfied and the operator $T$ will be called a Calderón-Zygmund operator (CZO).

We now give the definition of some operators closely related to the CZO. 
Given $b \in L^{1}_{\text{loc}}$ and  a linear operator $T$, the first order commutator of $T$ is formally defined by 
\begin{equation*}
    T_{b} f = [bT,f]= bTf - T(bf).
\end{equation*}
The higher order commutators of $T$ are defined, recursively, by
\begin{equation*}
    T^{m}_{b}f=[b, T^{m-1}_{b}]f, \hspace{5mm} \hspace{5mm} m \in \mathbb{N}.
\end{equation*} 
We say that a function  $b$ belongs to the Bounded Mean Oscillation space, BMO, if the quantity
\begin{equation*}
    \|b \|_{BMO}= \sup_{Q}\frac{1}{|Q|} \int_{Q} |b(x) -b_{Q}| \,dx
\end{equation*} is finite, where $\displaystyle b_{Q} = \frac{1}{|Q|} \int_{Q} b(x) \,dx$.

\vspace{7mm}
We are now in position to state our main theorems. 
 The first result establishes a generalization of the strong Fefferman-Stein estimates proved in \cite{P94} and \cite{P97}. This inequality allows us to prove Theorem \ref{Teo FS} and Theorem \ref{FS: FS para el conmutador}, which are improvements of the weighted mixed inequalities of Fefferman-Stein type proved in \cite{BCP22} and \cite{BPR25} for $T$ and $T_{b}^{m}$, respectively. 
\begin{teo}\label{FS: teo aux p}
   Let $m\in \mathbb{N}\cup \{ 0\}$, $\varepsilon>0$, $q>1$ and $1<p<\min \{q,1+\frac{\varepsilon}{m+1}  \}$. If $v \in RH_{\infty} \cap A_{q'}$, then the inequality
\begin{equation}\label{FS: aux1, 1}
    \int_{\mathbb{R}^n} |T_{b}^{m}f(x)|^{p} w(x) v^{1-p}(x) \, dx \leq C \int_{\mathbb{R}^n} |f(x)|^{p} M_{\Phi_{m+\varepsilon}, v^{1-q}}w(x) v^{1-p}(x) \,dx,   \end{equation} holds for every non-negative and locally integrable function $w$.
\end{teo}
If $m=0$, Theorem \ref{FS: teo aux p}
 allows us to obtain the following weighted mixed inequality of Fefferman-Stein type for a $T$ a CZO.
\begin{teo}\label{Teo FS}
Let $\varepsilon>0$ and $q>1$. If $v \in RH_{\infty} \cap A_{q'}$, then the inequality
    \begin{equation*}
         uv\left( \left\{ x \in \mathbb{R}^n : \frac{|T(fv)(x)|}{v(x)} > \lambda \right\} \right) \leq \frac{C}{\lambda} \int_{\mathbb{R}^{n}} |f(x)| M_{\Phi_{\varepsilon}, v^{1-q}} u (x) v(x) \,dx
    \end{equation*} holds for every $\lambda >0$ and every non-negative and locally integrable function  $u$.
\end{teo}  
\noindent The result above improves inequality \eqref{BCP que mejoramos} obtained in \cite{BCP22}, since $v \leq M(\Psi(v))$.

If $m\geq 1$, Theorem \ref{FS: teo aux p} leads us to obtain the corresponding mixed inequality of Fefferman-Stein type for commutators of CZO with BMO symbols, $T_{b}^{m}$.

\begin{teo}\label{FS: FS para el conmutador}
Let $m\in \mathbb{N}$, $\varepsilon>0$, $q>1$ and $b \in \text{BMO}$. If $v \in RH_{\infty} \cap A_{q'}$, then the inequality
    \begin{equation*}
           uv\left( \left\{ x \in \mathbb{R}^n : \frac{|T_{b}^{m}(fv)(x)|}{v(x)} > \lambda \right\} \right) \leq C \int_{\mathbb{R}^{n}} \Phi_{m
         }\left( \|b \|_{\text{BMO}}^{m} \frac{|f(x)| }{\lambda}\right) M_{\Phi_{m+\varepsilon}, v^{1-q}} u (x) v(x) \,dx
    \end{equation*} holds for every $\lambda >0$ and every non-negative and locally integrable function   $u$.
\end{teo}

Theorem \ref{FS: FS para el conmutador} also improves the inequality obtained in \cite{BPR25} for $T_{b}^{m}$. 

The proofs of Theorem \ref{Teo FS} and \ref{FS: FS para el conmutador} rely on the Calderón-Zygmund decomposition with respect to the measure $d\mu(x)=v(x)\,dx$. When $v=1$, theorems \ref{Teo FS} and \ref{FS: FS para el conmutador} give the weak Fefferman-Stein inequalities obtained in  \cite{P97} and \cite{PP01}, respectively.
\vspace{5mm}

This article is organized as follows: in Section \ref{preliminares} we present basic definitions and properties required to prove the main theorems. Section \ref{auxiliares} is devoted to state and prove some auxiliary results and Section \ref{principales} contains the proofs of the results established in Section \ref{section 1}.

\section{Preliminaries and definitions}\label{preliminares}
By a weight we understand a locally integrable function $w$ that is positive almost everywhere. The Muckenhoupt $A_{1}$ class is the collection of weights $w$ such that there exists a positive constant $C$ that verifies
\begin{equation}\label{claseA1}
    \frac{1}{|Q|} \int_{Q} w(x)\,dx \leq C \inf_{Q}w
\end{equation} for every cube $Q \subset \mathbb{R}^{n}$ with sides parallel to the coordinate axes.
If $1<p< \infty$, the Muckenhoupt $A_p$ consists on the weights $w$ that verify
\begin{equation}\label{claseAp}
    \left( \frac{1}{|Q|} \int_{Q}w(x) \,dx \right) \left( \frac{1}{|Q|} \int_{Q} w(x)^{1-p'}\,dx \right)^{p-1} \leq C
\end{equation} for some positive constant $C$ and for every cube $Q \subset \mathbb{R}^{n}$. As usual, we denote $\displaystyle A_{\infty}= \cup_{p \geq 1}A_p$.

We shall also deal with classes of weights with respect to a measure involving another weight. Given a weight $u$, we say that $w \in A_{1}(u)$ if it verifies that 
\begin{equation*}
    \frac{1}{u(Q)} \int_{Q} w(x) u(x)\,dx \leq C \inf_{Q}w
\end{equation*} for some positive constant $C$ and every cube $Q \subset \mathbb{R}^n$. For $1<p<\infty$, we say that $w \in A_{p}(u)$ if it verifies that
\begin{equation*}
    \left( \frac{1}{u(Q)} \int_{Q}w(x) u(x)\,dx \right) \left( \frac{1}{u(Q)} \int_{Q} w(x)^{1-p'} u(x)\,dx \right)^{p-1} \leq C
\end{equation*} for some positive constant $C$ and every cube $Q \subset \mathbb{R}^n$.

 Given $1<s<\infty$ we say that $w$ belongs to the reverse Hölder class $RH_s $ if there exists a positive constant $C$  such that  
\begin{equation}\label{RHs}
    \left( \frac{1}{|Q|} \int_{Q} w(x)^{s} \,dx  \right)^{1/s} \leq \frac{C}{|Q|} \int_{Q} w(x) \,dx
\end{equation} holds for every cube $Q \subset \mathbb{R}^{n}$. We also shall say that $w \in RH_{\infty}$ if there exists a positive constant $C$ such that the inequality 
\begin{equation}\label{RHinfty}
    \sup_{Q} w \leq \frac{C}{|Q|} \int_{Q}w(x) \,dx
\end{equation} holds for every $Q \subset \mathbb{R}^{n}$. It is well-known that every $A_p$ weight satisfies a reverse Hölder inequality. See for example \cite{Javi} and \cite{grafakos}.

We now establish some useful properties of the classes of weights mentioned above. The proofs can be found in \cite{CF74} and \cite{CU-N-Reverse}. 
\begin{propo}\label{Propiedades de pesos}
    Let $v$ be a weight, then the following statements hold.
    \begin{enumerate}[\rm(i)]
        \item $RH_{\infty} \subset A_{\infty}$. 
         \item \label{v a la q en rhinf}  If $v \in RH_{\infty} $,  then $v^{\varepsilon} \in RH_{\infty} $ for every $\varepsilon >0$.
        \item  If $v \in RH_{\infty} \cap A_{q'} $, then $v^{1-q} \in A_{1} $ for $1<q<\infty$.
        \item \label{v a la delta en A1} If $v \in A_1$, then $v^\delta \in A_{1} \text{ for } 0 < \delta <1$. 
\end{enumerate} \end{propo}

\vspace{5mm}
Recall that $\Phi$ is a Young function if there exists a non-zero, non-negative and increasing function $\phi$ that verifies $\displaystyle \lim_{s\to \infty} \phi(s) = \infty$ and $\displaystyle\Phi(t) = \int_{0}^{t} \phi(u) \,du$. From these properties, we obtain that \begin{equation*}
    \frac{1}{2} \,\phi\left(\frac{t}{2}\right) \leq \frac{1}{t} \int_{t/2}^{t}\phi(s) \,ds \leq \frac{1}{t} \int_{0}^{t} \phi(s)\,ds \leq \phi(t),
\end{equation*} which means that
\begin{equation}\label{Young: phi y su derivada}
    \frac{1}{2} \phi\left(\frac{t}{2}\right) \leq \frac{\Phi(t)}{t} \leq \phi(t).
\end{equation}
We say that a Young function $\Phi$ is doubling if it satisfies that $\Phi(2t) \leq C \Phi(t)$.

Given two Young functions $\Phi$ and $\Psi$ we say that $\Phi$ dominates $\Psi$ at infinity, and we write $\Psi \prec_{\infty} \Phi$, if there exist positive constants  $a,b, t_0$ such that 
 \begin{equation*}
     \Psi(t) \leq b \, \Phi (at) \hspace{ 7mm} \text{for every } t \geq t_0.
 \end{equation*}
We write $\Phi \approx_{\infty} \Psi$ to mean that  $\Psi \prec_{\infty} \Phi$ and  $\Phi \prec_{\infty} \Psi$. 
\begin{obs}\label{funcion de young domina a la identidad}
     It is easy to see that if $t \geq 1$, then 
    \begin{equation*}
        t \leq \frac{1}{\Phi(1)} \Phi(t).
    \end{equation*} This means that every Young function $\Phi$ dominates the identity function at infinity.
\end{obs}

Given a Young function $\Phi$ and $1\leq p < \infty$, we say that $\Phi$ is a $p\,$-Young function if $\Psi(t)=\Phi(t^{1/p})$ is also a Young function. As an example, the function $\Phi(t)=t^{p}(1+\log^{+} t)^{\varepsilon}$, for $p\geq 1$ and $\varepsilon \geq 0$, is a $p\,$-Young function.

\begin{propo}\label{Young no decreciente}
    Let $1\leq  p<\infty$. If $\Phi$ is a $p\,$-Young function and $\Theta(t)=t^{p}$, then $\Theta \prec_{\infty} \Phi$. Furthermore, $\Phi(t)/\Theta(t)$ is non decreasing.
\end{propo}
\begin{proof}    
By definition of $\Phi$, $\Psi(t) = \Phi(t^{1/p})$ is a Young function and, by Remark \ref{funcion de young domina a la identidad}, there exists a positive constant $C$ such that $t \leq C \Psi(t)$ for $t \geq 1$. Therefore, we have that 
\begin{align*}
       \Theta(t) = t^{p} \leq C \Psi (t^{p}) = C \Phi(t) \hspace{7mm} t \geq 1,
    \end{align*} 
which proves that $\Theta \prec_{\infty} \Phi$.

We shall now see that $\Phi(t)/\Theta(t)$ is non decreasing. In fact,
\begin{equation*}
   \left( \frac{\Phi(t)}{\Theta(t)}\right)'= \left(\frac{\Psi(t^p)}{t^p}\right) ' = \frac{\psi(t^p) p t^{p-1} t^p - \Psi(t^p) p t^{p-1}}{t^{2p}},
\end{equation*} where $\psi= \Psi'$. The last expression is non-negative if  $$\psi(t^p) t^{p} - \Psi(t^p) \geq 0. $$ By inequality \eqref{Young: phi y su derivada}, we have that
\begin{equation*}
    \frac{ \Psi(t^p)}{t^p} \leq  \psi(t^p),
\end{equation*} which finalizes the proof. 
 \end{proof}

The following propositions establish some useful properties of the Luxemburg averages. We give the proofs for the sake of completeness. 
\begin{propo}\label{Young: igualdad normas a la r}
   Given $r>0$, a weight $w$ and a Young function $\Phi$, we have that \begin{equation*}
        \|f^{r} \|_{\Phi,Q,w} = \|f \|_{\Psi,Q,w}^{r},
    \end{equation*}
    where $\Psi(t)=\Phi(t^{r})$. Particularly, if $r \geq 1$ and $\Phi$ is a $p$-Young function, $p\geq 1$, then
    \begin{equation*}
         \|f^{r/p} \|_{\Phi,Q,w} = \|f \|_{\Psi,Q,w}^{r/p},
    \end{equation*}
    where $\Psi(t)=\Phi(t^{r/p})$ is a $r$-Young function.
\end{propo}
\begin{proof}
Observe that
   \begin{equation*}
   \begin{split}
    \|f^{r} \|_{\Phi,Q,w} &= \inf \left\{ \lambda >0 : \frac{1}{w(Q)}\int_{Q} \Phi\left( \frac{|f|^{r}}{\lambda} \right)w\leq 1 \right\} \\
    & =\inf \left\{ \alpha^{r} >0 : \frac{1}{w(Q)}\int_{Q} \Psi\left( \frac{|f|}{\alpha} \right) w \leq 1 \right\} \\
    &=  \|f \|_{\Psi,Q,w}^{r},
   \end{split}
   \end{equation*} which proves the first equality.

In order to obtain the second identity, by using the estimation above with $r$ replaced by $r/p$, we obtain that
\begin{equation*}
 \| f^{r/p}\|_{\Phi,Q,w} = \|f \|_{\Psi,Q,w}^{r/p}. \qedhere
\end{equation*}
\end{proof}
\begin{propo}\label{dominancia y prom lux}
    Given a weight $w$ and two Young functions $\Phi$ and $\Psi$ such that $\Phi \prec_{\infty} \Psi$, there exists a positive constant $C$, depending only on $\Phi$ and $\Psi$, such that the inequality
    \begin{equation*}
        \|f \|_{\Phi,Q,w} \leq C \|f \|_{\Psi,Q,w}
    \end{equation*} holds for every cube $Q$ and every  $f$.
\end{propo}
\begin{proof}
We can suppose that $\|f\|_{\Psi, Q,w} < \infty$, otherwise the result is trivial. Let $a,b$ and $t_0$ be such that $\Phi(t) \leq b \Psi(at)$ for every $t \geq t_0$. Fix a cube $Q$ and consider the set
    \begin{equation*}
        I = \{x \in Q : |f(x)| \leq a t_0 \| f\|_{\Psi,Q,w} \}.
    \end{equation*}
We have that
\begin{equation*}
\begin{split}
    \frac{1}{w(Q)} \int_{Q} \Phi \left( \frac{|f(x)|}{a \| f\|_{\Psi,Q,w}} \right) w(x)\,dx & = \frac{1}{w(Q)} \int_{Q\cap I} \Phi \left( \frac{|f(x)|}{a \| f\|_{\Psi,Q,w}} \right) w(x)\,dx \\ & \hspace{10mm}+\frac{1}{w(Q)} \int_{Q\backslash I} \Phi \left( \frac{|f(x)|}{a \| f\|_{\Psi,Q,w}} \right) w(x)\,dx \\
    &\leq \Phi(t_0) + \frac{b}{w(Q)} \int_{Q\backslash I} \Psi \left( \frac{|f(x)|}{ \| f\|_{\Psi,Q,w}} \right) w(x)\,dx \\ 
    & \leq \Phi(t_0) +b.
    \end{split}
\end{equation*} 
If $\Phi(t_0) +b \leq 1$, it follows that
\begin{equation*}
    \| f\|_{\Phi, Q,w} \leq a \| f\|_{\Psi,Q,w}.
\end{equation*} If $\Phi(t_0) +b >1$, then by the convexity of $\Phi$ we obtain that
\begin{equation*}
    \| f\|_{\Phi, Q,w} \leq a(\Phi(t_0) +b ) \| f\|_{\Psi,Q,w}.\qedhere
\end{equation*}
\end{proof}
\begin{obs}\label{dominancia y maximales}
    Proposition \ref{dominancia y prom lux} implies that if $\Phi \prec_{\infty} \Psi$, then $M_{\Phi,w}f \leq C M_{\Psi,w}f$ almost everywhere.
\end{obs}

The following generalized Hölder inequality for Young functions will be useful for our main purposes. The proof can be found in \cite{ONeil65}.

\begin{teo}\label{Holder gen}
    Let $w$  be a weight and $\Phi, \Psi $ and $ \Theta$ Young functions such that 
\begin{equation}\label{condicion triplete holder}
        \Phi^{-1}(t) \Psi^{-1}(t) \leq C \Theta^{-1}(t) \hspace{7mm} \text{for } t \geq t_0 \geq 0.
    \end{equation} Then, the inequality 
    \begin{equation*}
        \|fg \|_{\Theta, Q, w} \leq C \| f\|_{\Phi,Q,w} \|g\|_{\Psi,Q,w}
    \end{equation*} holds for every cube $Q \subset \mathbb{R}^n$.
\end{teo}

Let $\Phi$ be a Young function, we define its complementary function  $\tilde{\Phi}$ by
\begin{equation*}
    \tilde{\Phi}(t) = \sup \{ ts - \Phi(s) : s \geq 0 \}.
\end{equation*}  For example, if $\Phi_{\varepsilon}(t)= t (1+\log^{+} t)^{\varepsilon}$, then $\tilde{\Phi}_{\varepsilon}(t) \approx e^{t^{1/\varepsilon}}-1$.

If $\Phi$ and $\tilde{\Phi}$ are Young functions, the relation
\begin{equation*}
    \Phi^{-1} (t) \tilde{\Phi}^{-1}(t) \approx t
\end{equation*} holds for every positive $t$ (for more details see \cite{KR} and \cite{raoren}). From Theorem \ref{Holder gen} it follows that 
\begin{equation*}
\int_{Q} |f(x)g(x)| w(x) \,dx  \leq C \| f\|_{\Phi,Q,w} \|g\|_{\tilde{\Phi},Q,w}.
\end{equation*}

 Given $1 <p < \infty$ and a Young function $\Phi$ we say that $\Phi $ satisfies a $B_{p}$ condition ($\Phi \in B_{p}$) if there exists a positive constant  $a$ such that
 \begin{equation*}
     \int_{a}^{\infty} \frac{\Phi(t)}{t^{p}} \frac{dt}{t} < \infty.
 \end{equation*} 

For example the family of Young functions $\Phi(t)=t^{r} (1+\log^{+}t)^{\varepsilon}$, with $r \geq 1$ and $\varepsilon \geq 0$, belongs to $B_p$ for every $p > r$. 

We now present some useful properties of the $BMO$ space of functions.

\begin{propo}\label{BMO comparacion norma phi y bmo}
    Let $b \in BMO$, $\delta \geq 1$ and $\Psi(t)= e^{t^{1/\delta}}-1$. Then there exists a positive constant $C$ such that the inequality
    \begin{equation*}
        \|b-b_{Q} \|_{\Psi,Q} \leq C \|b \|_{\text{BMO}}    \end{equation*} holds for every cube $Q$.
\end{propo}

The proof is a consequence of the John-Nirenberg inequality (see \cite{JN61}).

\begin{propo}[\cite{P95}]\label{BMO valor abs con promedios en 2k Q}
    Given $b \in BMO$, there exists a positive constant $C$ such that
\begin{equation*}
    |b(x) - b_{2^{k}Q}| \leq C k \| b\|_{BMO}
\end{equation*} for every $k \in \mathbb{N}$ and every cube $Q$.
\end{propo}

\section{Some previous estimates}\label{auxiliares}
In this section we present several results that will be useful in the proofs of our main theorems.

\begin{teo}\label{prop maximal gen 1}
Let $\Phi$ be a Young function, $w$ a doubling weight and $Q$ a cube. Then there exists a positive constant $C$ such that the inequality 
\begin{equation*}
    w(\{ x \in Q : M_{\Phi,w}f(x) > \lambda \}) \leq C \int_{Q} \Phi\left( \frac{|f(x)|}{\lambda} \right) w(x)\,dx
\end{equation*} holds for every $\lambda >0$.
As a consequence 
\begin{equation}\label{debil 2young }
   w(\{ x \in Q : M_{\Phi,w}f(x) > \lambda \}) \leq C \int_{\{x: |f(x)|> \lambda/2 \}} \Phi\left(\frac{2|f(x)|}{\lambda} \right) w(x) \, dx.
\end{equation}\end{teo}
\begin{proof}
   Consider $\Omega=\{ x \in Q : M_{\Phi,w}f(x) > \lambda \}$. If $x \in Q$ and $\|f \|_{\Phi,Q,w} >\lambda$, then $M_{\Phi,Q,w}f(x) >\lambda$, and we get that $\Omega =Q$. Since
    \begin{equation*}
        \frac{1}{w(Q)} \int_{Q} \Phi \left( \frac{|f(x)|}{\lambda} \right)w(x) \, dx >1,
    \end{equation*} we obtain that
    \begin{equation*}
        w(\Omega) = w(Q) \leq \int_{Q} \Phi \left( \frac{|f(x)|}{\lambda} \right) w(x)\, dx.
    \end{equation*}

    If $\|f\|_{\Phi,Q,w} \leq \lambda$, we consider the decomposition of $Q$ in dyadic subcubes and take the collection of maximal disjoint subcubes  $\{Q_j \}_{j}$ satisfying
    \begin{equation*}
        \|f \|_{\Phi,Q_j, w} > \lambda
    \end{equation*} and $\Omega = \bigcup_{j}Q_j$.
   Therefore, 
    \begin{equation*}
        \frac{1}{w(Q_j)} \int_{Q_j }\Phi \left( \frac{|f(x)|}{\lambda} \right)w(x) \,dx \geq 1
    \end{equation*}
  and it follows that
    \begin{equation*}
        w(\Omega)=\sum_{j}w(Q_j) \leq \sum_{j} \int_{Q_j} \Phi \left( \frac{|f(x)|}{\lambda} \right) w(x)\,dx \leq \int_{Q}\Phi \left( \frac{|f(x)|}{\lambda} \right)w(x) \,dx.
    \end{equation*}

    In order to prove \eqref{debil 2young }, we write $f_{1}=f \chi_{\{|f|>\lambda/2\}}$ and $f=f_1 + f_2$. Using the first part, we obtain that
\begin{equation*}
    \begin{split}
      w(\{ x \in Q : M_{\Phi,w}f(x) > \lambda \}) & \leq  w(\{ x \in Q : M_{\Phi,w}f_{1}(x) > \lambda/2 \}) + w(\{ x \in Q : M_{\Phi,w}f_{2}(x) > \lambda/2\}) \\ & 
          \leq C \int_{\{x \in Q : |f(x)|> \lambda/2 \}} \Phi\left(\frac{2|f(x)|}{\lambda} \right)w(x) \, dx 
    \end{split}
\end{equation*} since $M_{\Phi,w}f_2 \leq \lambda/2$ and so the term involving this function is equal to zero.
\end{proof}
\vspace{5mm}

Given a Young function $\Phi$, we define 
\begin{equation}\label{Young sub 0}
        \Phi_{0}(t) = \left\{\begin{array}{cc} 
            0 & 0\leq t <1, \\
           \Phi(t) - \Phi(1) &  t \geq 1.
        \end{array}\right.
    \end{equation}

    If we define $\phi_0(t)=\phi(t)$ for $t\geq 1$ and $\phi_0(t)=0$ otherwise, we get that $\displaystyle \Phi_0(t)=\int_0^t\phi_0$, so $\Phi_0$ is also a Young function.

    \begin{obs}\label{FS obs phi y phi_0}
    For every Young function $\Phi$ we have that $\Phi\approx_\infty \Phi_0$, where $\Phi_0$ is defined as in \eqref{Young sub 0}. It is clear that $\Phi_0 \prec_{\infty} \Phi$ by definition. 
    We now see that $\Phi \prec_{\infty} \Phi_{0}$. We first claim that $\Phi (1) \leq \Phi_{0}(at) $ for $a >1$ and $t> 2/a$. Indeed, since $\phi$ is increasing, by changing variables we obtain 
\begin{equation*}
        \Phi(1) = \int_{0}^{1} \phi(v) \,dv  
        \leq\int_{0}^{1} \phi(v+1) \,dv 
         = \int_{1}^{2} \phi(u) \,du =\Phi(2) - \Phi(1)= \Phi_{0}(2).
\end{equation*}
Since $\Phi_{0}$ is increasing, we have that $\Phi_{0}(2) \leq \Phi_{0}(at)$ for $t> 2/a$, which proves the claim. 

Let us now consider $a>1$ and $t_{0}= 2/a$. Since $at>2$ for $t\geq t_0$, from the claim above we get
\begin{equation*}
        \Phi(t)  \leq \frac{1}{a} \Phi(at) 
         = \frac{1}{a} \left( \Phi_{0}(at) + \Phi(1) \right) 
        \leq \frac{2}{a} \Phi_{0}(at),
\end{equation*} which proves that $\Phi \prec_{\infty} \Phi_{0}$.
\end{obs}
    
The following result establishes a pointwise relation between generalized maximal operators. The proof is a slight modification of a result given in  \cite{anacomposition}.

\begin{teo}\label{FS: lema des aux tres maximales}
    Let $\Phi$ and $\Psi$ be Young functions and let 
    \begin{equation*}
        \Theta(t)=\int_{1}^{t} \Psi_{0}' (u) \Phi(t/u) \,du,
    \end{equation*} where $\Psi_{0}$ is defined as in  \eqref{Young sub 0}, then $\Theta$ is a Young function. Moreover, for every Young function $\Tilde{\Theta}$ such that $\Theta \prec_{\infty} \tilde{\Theta} $ we get 
    \begin{equation}\label{FS: des aux tres maximales}
        M_{\Psi,w}(M_{\Phi,w}f)(x) \leq C M_{\tilde{\Theta},w}f(x) \hspace{7mm} \text{ a.\,e.}
    \end{equation} 
\end{teo}

\begin{proof}
Without loss of generality, we can consider that $\Psi(1)=\Phi(1)=1$. Since $\Phi$ is a Young function, $ \displaystyle\Phi(t) = \int_{0}^{t} \phi(s) \,ds$, where $\phi$ is non negative and increasing. By changing variables, we obtain that
\begin{equation*}
    \Phi(t/u) = \frac{1}{u} \int_{0}^{t} \phi(v/u) \,dv.
\end{equation*}
Therefore,
\begin{align*}
  \Theta(t) = & \int_{1}^{t} \Psi_{0}'(u) \left( \frac{1}{u} \int_{0}^{t} \phi(v/u) \,dv \right) \,du  \\
  & = \int_{0}^{t} \int_{1}^{t} \Psi_{0}'(u) \frac{1}{u} \phi(v/u) \,du \, dv  = \int_{0}^{t} \theta(v) \,dv
\end{align*}
    where $\theta( v)=\int_{1}^{t} \Psi_{0}'(u)\phi(v/u) u^{-1}\,du$. The functions $\Psi_{0}'$ and $\phi$ are non negative and increasing, and so is $\theta$, which implies that  $\Theta$ is a Young function. 

Let us observe that if we prove the inequality
\begin{equation}\label{FS des tres maximales eq aux 1}
M_{\Psi_{0}, w}(M_{\Phi,w} f)(x) \leq C M_{\Theta,w}f(x),
\end{equation}
by Remarks \ref{dominancia y maximales} and \ref{FS obs phi y phi_0} we obtain 
    \begin{equation*}
        M_{\Psi,w}(M_{\Phi,w} f)(x) \leq C M_{\Psi_{0}, w}(M_{\Phi,w} f)(x) \leq C M_{\Theta,w}f(x) \leq C M_{\tilde{\Theta},w}f(x) ,
    \end{equation*} which proves \eqref{FS: des aux tres maximales}.
    
    In order to prove  \eqref{FS des tres maximales eq aux 1} it will be enough to show that there exists a positive constant $C$  such that $\| M_{\Phi,w} f  \|_{\Psi_{0},Q,w} \leq C \|f \|_{\Theta, Q,w}$ for every cube $Q$ that contains $x$. 
    
  We can assume $\| f\|_{\Theta, Q,w} = 1/2$, which means that
 \begin{equation}\label{lema: norma f 1/2}
     \frac{1}{ w(Q)} \int_{Q} \Theta (2|f(x)|) \,dw \leq 1
 \end{equation} also holds and it follows that
\begin{equation*}
\left\|M_{\Phi,w}f\right\|_{\Psi_{0},Q,w} \leq C \|f\|_{\Theta,Q,w}.
\end{equation*}

 By the definition of $\Psi_{0}' $ and applying Theorem \ref{prop maximal gen 1}, we get that
 \begin{align*}
 \begin{split}
     \int_{Q} \Psi_{0} (M_{\Phi,w}f(x)) \,w(x)\,dx = & \int_{0}^{\infty} \Psi_{0}'(s) \, w(\{x \in Q : M_{\Phi,w} f (x)> s \}) \,ds \\
      = & \int_{1}^{\infty} \Psi_{0}'(s) \,w(\{x \in Q: M_{\Phi,w}f (x) > s \}) \,ds\\ 
      \leq &  C \int_{1}^{\infty} \Psi_{0}'(s)\, \int_{\{ x \in Q :| f(x)| > s/2 \} } \Phi(2|f(x)| / s) \,w(x)\,dx \,ds  \\
        =& C \int_{Q} \int_{1}^{2|f(x)|} \Psi_{0}'(s) \Phi(2|f(x)| /s) \,ds \,w(x)\,dx \\
        = &  C \int_{Q} \Theta(2|f(x)|) \, w(x)\,dx \leq C w(Q),\end{split}
    \end{align*} where in the last inequality we have used \eqref{lema: norma f 1/2}. Therefore,  
\begin{align*}
    \frac{1}{w(Q)} \int_{Q} \Psi_{0} (M_{\Phi,w}f(x)) \,w(x)\,dx  \leq C,
\end{align*} where we can take $C>1$. It follows that 
\begin{equation*}
\|M_{\Phi,w}f\|_{\Psi_{0}, Q, w} \leq C = 2C \|f \|_{\Theta,Q, w}. \qedhere
\end{equation*}
\end{proof}

The following theorem is a generalization of the well-known fact that $(Mf)^\delta$ is an $A_1$ weight for $0<\delta<1$. For the Lebesgue measure, a proof can be found in  \cite{CUMPlibro}. 
\begin{teo} \label{FS: max a la delta en A1}
Let $\Phi$ be a Young function and $0<\delta<1$. Let $f$ be such that $M_{\Phi, w}f(x) <\infty$ a.e. with respect to the measure $d\mu(x)= w(x) \,dx$. Then 
 $(M_{\Phi,w}f)^{\delta} \in A_{1}(w)$.
\end{teo}
\begin{proof}
Let us first suppose that $\Phi$ is an  $r$-Young function, with $r>1$. Then, by Proposition \ref{Young no decreciente}$, \Phi(t)/t^r$ is non decreasing. Let $\Psi(t) = t^s$, with $ 1\leq s < r$, then if $\Theta$ is defined as in Theorem \ref{FS: lema des aux tres maximales} we get that 
\begin{align*}
    \Theta(t) = s \int_{1}^{t} \Phi(t/u) u^{s-1} \,du \leq C \int_{1}^{t} \frac{\Phi(t/u)}{(t/u)^{r}}u^{s-1} (t/u)^{r }\,du\leq C \Phi(t) \int_{1}^{t} u^{s-r-1} \,du\leq C \Phi(t).
\end{align*}
Thus $\Theta \prec_{\infty} \Phi$ and, consequently, Theorem \ref{FS: lema des aux tres maximales} implies that
\begin{equation*}
    M_{\Psi,w}(M_{\Phi,w}f)(x) \leq C M_{\Phi,w}f(x) 
\end{equation*} almost everywhere. Since $\Psi(t)=t^s$, it follows that 
\begin{align*}
    M_{w}((M_{\Phi,w}f)^{s}) = M_{\Psi,w}(M_{\Phi,w}f)^{s} \leq C M_{\Phi,w} f(x) ^{s}.
\end{align*}
This last inequality proves that  $(M_{\Phi,w}f)^{s} \in A_{1}(w)$ for $1 \leq s <r$. From item \eqref{v a la delta en A1} of Proposition \ref{Propiedades de pesos} we obtain that $(M_{\Phi,w}f)^{s} \in A_{1}(w)$ for every $0<s<r$.

Now consider a general Young function $\Phi$,  $r= 1/\delta >1$ and $\Psi(t)=\Phi(t^r) $, which is an $r$-Young function. From the case proved above applied to $f^{\delta}$ with $s=1$, it follows that $ M_{\Psi,w} \left(f^{\delta}\right) \in A_{1}(w)$. So, Proposition \ref{Young: igualdad normas a la r} leads to  $(M_{\Phi,w} f) ^{\delta} = M_{\Psi,w} (f^{\delta}) \in A_{1}(w)$, which completes the proof.  
\end{proof} 

The next lemma establishes a well-known bound for $L\, \log L$ type functions by means of power functions. A proof can be found in \cite{B19}.

\begin{lema}\label{L log L propiedad epsilon}
    Let $r\geq 1$, $\varepsilon\geq 0$ and  $\Phi(t) = t^{r} (1+ \log^{+}t)^{\varepsilon}  $. For $t \geq 1$ and $\delta > 0$ we get that
    \begin{equation*}
        \Phi(t) \leq C t^{r+\delta},
    \end{equation*} where $C=\max \left\{ (\varepsilon / \delta)^{\varepsilon}, 1\right\} $.
\end{lema}

The following lemma is useful in the proof of Theorem \ref{FS: teo aux p} and the main results.
\begin{lema}\label{FS: lema afirmacion clave}
    Let $\varepsilon\geq 0$, $q>1$ and $\Phi_{\varepsilon}(t)=t(1+\log^{+} t)^{\varepsilon}$. If  $v \in RH_{\infty}\cap A_{q'}$, then for every $1<p<q$, there exists a positive constant $C$ such that the inequality  \begin{equation*}
        M_{\Phi_{\varepsilon}}(wv^{1-p}) (x) \leq C v^{1-p} (x)M_{\Phi_{\varepsilon}, v^{1-q}}w (x) 
\end{equation*} holds for every weight $w$.
\end{lema}

\begin{proof}
    Fix $x\in \mathbb{R}^{n}$ and a cube $Q$ containing $x$. We shall prove that
    \begin{equation}\label{FS: puntual Fabio aux}
        \| wv^{1-p}\|_{\Phi_{\varepsilon}, Q} \leq C \| w\|_{\Phi_{\varepsilon},Q, v^{1-q}}\,\, v^{1-p}(x),    \end{equation} and the desired estimate will follow by taking supremum over the cubes containing $x$.

Let $\lambda_{0}=\| w\|_{\Phi_{\varepsilon},Q, v^{1-q}}\, v^{1-p}(x)$. Note that by the hypothesis and Proposition \ref{Propiedades de pesos} we get $v^{1-q} \in A_1$. We consider $\delta= \frac{q-1}{p-1} -1 >0$, so that $(1-p)(1+\delta)=1-q$.  Since $\Phi$ is submultiplicative, by Lemma \ref{L log L propiedad epsilon} we get that
\begin{equation*}
    \begin{split}
   \frac{1}{|Q|} \int_{Q} \Phi_{\varepsilon}\left( \frac{w(y) v(y)^{1-p}}{\lambda_{0}} \right) \,dy  & \leq  \frac{1}{|Q|} \int_{Q} \Phi_{\varepsilon}\left( \frac{w(y)}{\| w\|_{\Phi_{\varepsilon},Q, v^{1-q}}} \right) \Phi_{\varepsilon}\left( \frac{v(y)^{1-p}}{\inf_{Q } v^{1-p}}\right) \,dy\\
      & \leq  \frac{C}{|Q|} \int_{Q} \Phi_{\varepsilon}\left( \frac{w(y)}{\| w\|_{\Phi_{\varepsilon},Q, v^{1-q}}} \right) \left( \frac{v(y)^{1-p}}{\inf_{Q} v^{1-p}}\right)^{1+\delta} \,dy
    \end{split}
\end{equation*} Therefore we get 
\begin{equation*}
    \begin{split}
       \frac{1}{|Q|} \int_{Q} \Phi_{\varepsilon}\left( \frac{w v^{1-p}}{\lambda_{0}} \right)  & \leq   C \frac{v^{1-q}(Q)}{|Q|} \left(\frac{1}{v^{1-q}(Q)}  \int_{Q} \Phi_{\varepsilon}\left( \frac{w(y)}{\| w\|_{\Phi_{\varepsilon},Q, v^{1-q}}} \right) v(y)^{1-q}\,dy \right) \frac{1}{\inf_{Q } v^{1-q}} \\
        & \leq C\inf_{Q} v^{1-q}\frac{1}{\inf_{Q} v^{1-q}} = C.
    \end{split}
\end{equation*}
Then we have that \eqref{FS: puntual Fabio aux} holds.
\end{proof}

\begin{lema}\label{FS: aux maximal con dos pesos en a infty}
Given $\varepsilon \geq 0$, $\Phi_{\varepsilon}= t (1+ \log^+  t)^{\varepsilon}$ and $v \in RH_{\infty} \cap A_{q'}$, then for every $1<p<q$,   $(M_{\Phi_{\varepsilon},v^{1-q}}w)^{1-p'} v \in A_{\infty}$.
\end{lema}
\begin{proof}
   We shall prove that $(M_{\Phi_{\varepsilon},v^{1-q}}w)^{1-p'} v$ belongs to some $A_{t}$ class, with $1<t <\infty$ to be determined.

We first estimate the expression 
\begin{equation}\label{FS: Aux 1, pte der}
   I= \left( \frac{1}{|Q|} \int_{Q}M_{\Phi_{\varepsilon},v^{1-q}}w(x)^{(1-p')(1-t')} v(x)^{1-t'}\,dx  \right)^{t-1}.
\end{equation}

Choosing $t>p'$, we get that $0<(1-p')(1-t')<1$ and $t'<p<q$. Let $\delta= (1-p')(1-t')$.  By Theorem \ref{FS: max a la delta en A1}, $(M_{\Phi_{\varepsilon},v^{1-q}}w)^{\delta} \in A_{1}(v^{1-q})$ and by Proposition \ref{Propiedades de pesos},  $v^{1-q} \in A_1$, so we have that
  \begin{align*}
    \begin{split}
        I&\leq\left( \frac{1}{v^{1-q}(Q)} \int_{Q}M_{\Phi_{\varepsilon},v^{1-q}}w (x)^{\delta} v(x)^{1-q} \,dx \right)^{t-1} \sup_{Q} v^{(q-t') (t-1)}\left( \frac{v^{1-q}(Q)}{|Q|}\right) ^{t-1}
        \\ &   \leq C \left(\inf_{Q} (M_{\Phi_{\varepsilon},v^{1-q}}w)^{\delta}\right)^{t-1}  \sup_{Q} v^{(q-t') (t-1)} \left(\inf_{Q} v^{1-q}\right)^{t-1}.
    \end{split}
\end{align*}
Therefore, 
\begin{align*}
    \begin{split}
       & \left( \frac{1}{|Q|} \int_{Q}  (M_{\Phi_{\varepsilon},v^{1-q}}w)^{1-p'} v \right) \left( \frac{1}{|Q|} \int_{Q} (M_{\Phi_{\varepsilon},v^{1-q}}w)^{(1-p')(1-t')} v^{1-t'}   \right)^{t-1} \\
       & \hspace{5mm}\leq C \left(\frac{1}{|Q|}\int_{Q} (M_{\Phi_{\varepsilon},v^{1-q}}w)^{1-p'} \right)\left(\inf_{Q} (M_{\Phi_{\varepsilon},v^{1-q}}w)^{\delta}\right)^{t-1}  \sup_{Q} v^{(q-t') (t-1)+1} \left(\inf_{Q} v^{1-q}\right)^{t-1} \\& \hspace{5mm}
        \leq \frac{C}{|Q|} \int_{Q} (M_{\Phi_{\varepsilon},v^{1-q}}w)^{(1-p')+\delta(t-1)}   \leq C,
    \end{split}
\end{align*}
where we have used that $\delta (t-1) = -(1-p')$  and $(q-t')(t-1)+1-(q-1)(t-1)= 0$.

We have proved that  $M_{\Phi_{\varepsilon},v^{1-q}}(w)^{1-p'}v \in A_t$ for $t> p'$ and therefore it belongs to $A_{\infty}$, which completes the proof.
\end{proof}

The following inequality establishes a control of CZO by means of the Hardy-Littlewood maximal function $M$, involving $A_\infty$ weights. This result was proved by Coifman in  \cite{Coifman}. Recall that $C_{0}^{\infty}(\mathbb{R}^n)$ is the set of infinitely differentiable functions in $\mathbb{R}^n$ with compact support.
\begin{teo}\label{Coifmann con medida de Lebesgue}
    Let $T$ be a CZO, then for $0<p<\infty$ and every $w \in A_{\infty}$, there exists a positive constant $C$ such that the inequality    \begin{equation*}
        \int_{\mathbb{R}^n} |Tf(y)|^{p} w(y) \, dy \leq C \int_{\mathbb{R}^n} Mf(y)^{p} w(y) \, dy 
    \end{equation*} holds for every $f \in C_{0}^{\infty}(\mathbb{R}^n)$.
\end{teo}
An analogous result for higher order commutator is established in the following theorem. A proof can be found in  \cite{P97}.
\begin{teo}\label{Carlos FS fuerte conmutador y maximal iterada}
    Let $0<p<\infty$, $b \in BMO$, $T$ a CZO and $w \in A_{\infty}.$ Then there exists a positive constant $C$ such that
    \begin{equation*}
\int_{\mathbb{R}^n}|T_{b}^{k}f(x)|^{p}w(x) \,dx \leq C \int_{\mathbb{R}^n}  M^{k+1} f(x)^{p} w(x) \,dx.
    \end{equation*}
\end{teo}
The following theorem establishes a characterization for the $L^p$ boundedness of $M_\Phi$ and it was proved by C. Pérez in \cite{P94}. Recall that $\tilde{\Phi}$ denotes the complementary function of $\Phi.$
\begin{teo}\label{FS: Caracterizacion Carlos Maximales}
Let $1<p<\infty$. Let $\Phi$ be a Young function and  $\Psi(t)=\tilde{\Phi}(t^{p'})$. Then the following statements are equivalent.
 
    \begin{enumerate}[\rm i)]
        \item $\Phi \in B_p$,
        \item there is a positive constant $C$ such that 
        \begin{equation*}
            \int_{\mathbb{R}^n} M_{\Psi}f(y)^{p} \,dy \leq C \int_{\mathbb{R}^n} |f(y)|^{p}\,dy,
\end{equation*}
    \item\label{teo Bp carlos item 3} there is a positive
    constant $C$ such that 
    \begin{equation*}
        \int_{\mathbb{R}^n} Mf(y)^{p} \frac{u(y)}{M_{\Phi}w(y)^{p-1}} \,dy \leq C \int_{\mathbb{R}^n} |f(y)|^{p} \frac{Mu(y)}{w(y)^{p-1}}\,dy \end{equation*}
        for every locally integrable function  $f$ and every weights $w$ and $u$.
 \end{enumerate}
\end{teo}

The next lemma can be found in \cite{BCP22}.
\begin{lema}\label{FS:Lema 2.5 Ana} 
    Let $\Phi$ be a Young function, $w$ a doubling weight, $f$ such that $M_{\Phi,w} f (x) < \infty$ almost everywhere and $Q$ be a fixed cube. Then, \begin{equation*}
M_{\Phi,w} (f\chi_{\mathbb{R}^{n}\backslash RQ } )(x) \approx M_{\Phi,w} (f\chi_{\mathbb{R}^{n} \backslash RQ}) (y) ,
\end{equation*} for every $x,y \in Q$, where $R=4 \sqrt{n}$.
\end{lema}

The next lemma establishes a pointwise relation between $M_\Phi$ and its weighted version. A proof can be found in \cite{BCP22}.
\begin{lema}\label{lema comparacion puntual maximal phi con peso}
    Given $w\in A_{1}$ and $\Phi$ a Young function, there exists a positive constant $C$ such that 
        \begin{equation*}
            M_{\Phi}f(x) \leq C M_{\Phi, w} f(x),
        \end{equation*} for every $f$ that verifies $M_{\Phi,w}f(x) < \infty$ a.e.
\end{lema}

\section{Proof of main results}\label{principales}
	\begin{proof}[Proof of Theorem \ref{FS: teo aux p}]
Let us first prove the case $m=0$. We take $\Phi_{\varepsilon}$ and $p$ as in the hypothesis and let $\nu=v^{1-q}$. By Lemma \ref{FS: aux maximal con dos pesos en a infty}, we get that $(M_{\Phi_{\varepsilon},\nu}w)^{1-p'} v$ is an $A_{\infty}$ weight.

In order to prove \eqref{FS: aux1, 1} we estimate the following equivalent inequality for $T^*$, the adjoint operator of $T$,
\begin{equation}\label{A}
     \int_{\mathbb{R}^{n}} |T^{*}f(x)|^{p'} M_{\Phi_{\varepsilon},\nu}w(x)^{1-p'}v(x) \,dx \leq C \int_{\mathbb{R}^n} |f(x)|^{p'} w(x)^{1-p'}v(x) \,dx.
\end{equation} 

Since $T^*$ behaves essentially like $T$, we can apply Theorem \ref{Coifmann con medida de Lebesgue} to this operator to get
\begin{align*}
    \int_{\mathbb{R}^{n}} |T^{*}f(x)|^{p'} M_{\Phi_{\varepsilon},\nu}w(x)^{1-p'}v(x) \,dx &\leq C \int_{\mathbb{R}^{n}} Mf(x)^{p'} M_{\Phi_{\varepsilon},\nu}w(x)^{1-p'}v(x) \,dx \\
    & \leq C \int_{\mathbb{R}^{n}} Mf(x)^{p'} M_{\Phi_{\varepsilon}}(wv^{1-p})(x)^{1-p'}(x) \,dx
\end{align*} where we have used Lemma \ref{FS: lema afirmacion clave} in the last inequality. So, in order to obtain \eqref{A}, we have to prove that 
\begin{equation*}
\int_{\mathbb{R}^n} Mf(x)^{p'}M_{\Phi_{\varepsilon}}(wv^{1-p})(x)^{1-p'} \,dx \leq C \int_{\mathbb{R}^n} |f(x)|^{p'} w(x)^{1-p'} v(x)\,dx.
\end{equation*} The last inequality is a consequence of \eqref{teo Bp carlos item 3} in Theorem \ref{FS: Caracterizacion Carlos Maximales}, since $\Phi_{\varepsilon}$ satisfies the $B_{p'}$ condition. 

We now consider the case $m>0$. As in the previous result, this is equivalent to prove that
    \begin{equation*}
       A=  \int_{\mathbb{R}^n} |T_{b}^{m,*}f(x)|^{p'} M_{\Phi_{m+\varepsilon}, v^{1-q}}w(x)^{1-p'} v(x) \,dx \leq C \int_{\mathbb{R}^n} |f(x)|^{p'} w(x)^{1-p'} v(x) \,dx.
    \end{equation*}

By Lemma \ref{FS: aux maximal con dos pesos en a infty}, Theorem \ref{Carlos FS fuerte conmutador y maximal iterada} and Lemma \ref{FS: lema afirmacion clave} we get that
\begin{equation*}
\begin{split}
       A &\leq C \int_{\mathbb{R}^n} M^{m+1}f(x)^{p'} M_{\Phi_{m+\varepsilon}, v^{1-q}}w(x)^{1-p'} v(x) \,dx \\
       &\leq C \int_{\mathbb{R}^n} M^{m+1}f(x)^{p'} M_{\Phi_{m+\varepsilon}}(wv^{1-p})(x)^{1-p'} \,dx.
\end{split}
\end{equation*} Therefore, we have to prove that 
\begin{align*}
    \int_{\mathbb{R}^n} M^{m+1}f(x)^{p'} M_{\Phi_{m+\varepsilon}}(wv^{1-p})(x)^{1-p'} \,dx \leq C \int_{\mathbb{R}^n} |f(x)|^{p'}w(x)^{1-p'}v(x) \,dx,
\end{align*}
which, by Remark \ref{obs maximal phi_m e iterada m+1}, is equivalent to
\begin{align*}
B=\int_{\mathbb{R}^n} M_{\Phi_{m}}(fw^{1/p}v^{-1/p'})(x)^{p'}  M_{\Phi_{m+\varepsilon}}(wv^{1-p})(x)^{1-p'} \,dx \leq C \int_{\mathbb{R}^n} |f(x)|^{p'} \,dx.
\end{align*} Using the fact that $\Phi^{-1}_{m}(t) \approx \frac{t}{(\log t)^{m}}$, by taking $\delta>0$ we get
\begin{equation*}
    \begin{split}
        \Phi_{m}^{-1}(t) \approx \frac{t}{(\log t)^{m}} & \approx \frac{t^{1/p}}{(\log t)^{m+\frac{p-1+\delta}{p}}} t^{1/p'}(\log t)^{\frac{p-1+\delta}{p}} \\
        & = \Psi^{-1}(t) \Theta^{-1}(t),
    \end{split}
\end{equation*}
where $\Psi(t)= t^{p} (\log t )^{(m+1)p-1 +\delta}$ and $\Theta(t) = t^{p'}(\log t)^{-(1+(p'-1)\delta)}$.

Using the generalized Hölder inequality (Theorem \ref{Holder gen})  we obtain that  
\begin{align*}
    M_{\Phi_{m}}(fw^{1/p} v^{-1/p'}) \leq C M_{\Theta} f \, \, M_{\Psi}(w^{1/p} v^{-1/p'}) \approx M_{\Theta} f \, \, (M_{\Phi_{(m+1)p-1+\delta}}(w v^{-p/p'}))^{1/p}.
\end{align*} Let $ \delta=m + \varepsilon - (m+1)p+1 >0$  we get that
\begin{align*}
    \begin{split}
B&  \leq C  \int_{\mathbb{R}^n}  M_{\Theta} f (x)^{p'}  \, M_{\Phi_{m+\varepsilon}}(w v^{-p/p'})(x)^{p'/p} M_{\Phi_{m+\varepsilon}}(wv^{1-p})(x)^{1-p'} \, dx \\
         &= C  \int_{\mathbb{R}^n}  M_{\Theta} f (x)^{p'}\,dx \\
         & \leq C \int_{\mathbb{R}^n} |f(x)|^{p'} \,dx,
    \end{split}
\end{align*}
Therefore, since $\Theta \in B_{p'}$ we obtain the desired estimate by applying Theorem \ref{FS: Caracterizacion Carlos Maximales}.
\end{proof}

\begin{proof}[Proof of Theorem \ref{Teo FS}] 
Fix $\lambda>0$ and assume, without loss of generality, that $f$ is a non-negative function with compact support. We consider the Calderón-Zygmund decomposition of $f$ at level $\lambda$ with respect of the measure $d\mu(x)=v(x) \,dx$, which is doubling since $v \in RH_{\infty} \subset A_{\infty}$. Then,  we obtain a sequence of dyadic disjoint cubes $\left\{ Q_{j}\right\}_{j}$ such that 
     \begin{enumerate}[\rm i)]
        \item $f(x) \leq \lambda$ for almost every $x \notin \bigcup_{j} Q_j$, 
        \item $\displaystyle \lambda <\frac{1}{v(Q_j)} \int_{Q_j} f(x) v(x) \,dx < C\lambda $.
    \end{enumerate}

    Let $\Omega= \bigcup_{j} Q_j$, $\displaystyle f_{Q_j}^{v}=\frac{1}{v(Q_j)} \int_{Q_j} f(x) v(x) \,dx$ and $f(x)=g(x) + h(x)$, where
    \begin{equation*}
        g(x) =\left\{ \begin{array}{c c}
            f(x) & \hspace{5mm} \text{if } x \notin \Omega \\
            f_{Q_{j}}^{v} &   \hspace{5mm} \text{if } x \in Q_j
        \end{array}\right.
    \end{equation*}
 and
 \begin{equation*}
     h(x) = \sum_{j} \left( f(x) -   f_{Q_{j}}^{v} \right) \chi_{Q_j}(x) = \sum_{j} h_j(x).
 \end{equation*}
It follows that $g(x) \leq C \lambda$ for almost every $x \in \mathbb{R}^n $ and 
\begin{equation}\label{FS prop h_j}
    \int_{Q_j} h_{j} (x) v(x) \,dx =0
\end{equation} for every $j$. We shall denote by $x_j$ and $\ell_j$  the centre and side length of $Q_j$, respectively. We also denote $Q_{j}^{*}=4\sqrt{n} \,Q_j$ and  $\Omega^{*}= \bigcup_{j} Q_{j}^{*}$. We have that 
\begin{equation*}
    \begin{split}
       uv\left( \left\{ x : \frac{|T(fv)(x)|}{v(x)} > \lambda \right\} \right) & \leq uv\left( \left\{ x \in \mathbb{R}^n \backslash \Omega^{*}: \frac{|T(gv)(x)|}{v(x)} > \frac{\lambda}{2} \right\} \right) + uv(\Omega^{*}) \\
       & \hspace{10mm} +uv\left( \left\{ x \in \mathbb{R}^{n }\backslash \Omega^{*}: \frac{|T(hv)(x)|}{v(x)} > \frac{\lambda}{2}
       \right\} \right) \\
       & = I_{1} + I_{2} + I_{3}.
    \end{split}
\end{equation*}

In order to estimate $I_1$, let us consider $1<p<\min \{ q, 1+\varepsilon\}$. By applying Tchebyshev inequality with and the hypothesis on the weights, from Theorem \ref{FS: teo aux p} and the fact that $g(x)\leq C\lambda$ a.e. we get
\begin{equation*}
    \begin{split}
        uv\left(  \left\{ x \in \mathbb{R}^{n}\backslash \Omega^{*} : \frac{|T(gv)(x)|}{v(x)} > \frac{\lambda}{2}\right\} \right) & \leq \frac{C}{\lambda^{p}} \int_{\mathbb{R}^n} |T(gv)(x)|^{p} u^{*}(x)v(x)^{1-p} \,dx \\
& \leq \frac{C}{\lambda^p} \int_{\mathbb{R}^n} g(x)^{p}v(x)^p M_{\Phi_{\varepsilon},v^{1-q}}u^{*}(x)\, v(x)^{1-p} \,dx  \\
& \leq \frac{C}{\lambda} \int_{\mathbb{R}^n} g(x) M_{\Phi_{\varepsilon},v^{1-q}}u^{*}(x)\, v(x) \,dx ,\\
\end{split}
\end{equation*} where $u^{*}= u \chi_{\mathbb{R}^{n}\backslash \Omega^{*}}$. By the definition of $g$,  \begin{equation*}
    \begin{split}
   I_1 \leq  \frac{C}{\lambda} \left( \int_{\mathbb{R}^n} f(x) M_{\Phi_{\varepsilon}, v^{1-q}}u^{*}(x) \,v(x) \,dx + \sum_{j}  f^{v}_{Q_j} \int_{Q_j} M_{\Phi_{\varepsilon}, v^{1-q}}u^{*}(x) \, v(x) \,dx \right).  
    \end{split}
\end{equation*} By Lemma \ref{FS:Lema 2.5 Ana}, for each $j$ we have that
\begin{equation*}
 \begin{split}
         f^{v}_{Q_j} \int_{Q_j} M_{\Phi_{\varepsilon}, v^{1-q}}u^{*}(x) v(x) \,dx &  \leq C f^{v}_{Q_j} v(Q_j) \inf_{Q_j}  M_{\Phi_{\varepsilon}, v^{1-q}}u^{*} \\
         & \leq C \int_{Q_{j}} f(x) M_{\Phi_{\varepsilon}, v^{1-q}}u^{*}(x) v(x) \,dx,
 \end{split}
\end{equation*}So we obtain 
\begin{align*}
    \begin{split}
        I_1 & \leq \frac{C}{\lambda}  \left( \int_{\mathbb{R}^{n} \backslash \Omega} f(x) M_{\Phi_{\varepsilon}, v^{1-q}}u^{*}(x) v(x) \,dx  + \sum_{j}   \int_{Q_{j}} f(x) M_{\Phi_{\varepsilon}, v^{1-q}}u^{*}(x) v(x) \,dx \right) \\ 
        & \leq \frac{C}{\lambda} \int_{\mathbb{R}^{n}} f(x) M_{\Phi_{\varepsilon}, v^{1-q}}u(x) v(x) \,dx
    \end{split}
\end{align*}

We now proceed with the estimation of $I_2$. Since $v\in RH_{\infty}$, by item \eqref{v a la q en rhinf} of Proposition \ref{Propiedades de pesos}, $v^{q} \in RH_{\infty}$. Since $t \leq  \Phi_{\varepsilon}(t)$, it follows that 
\begin{equation*}
    \begin{split}
        uv(Q_{j}^{*}) &= uv^{1-q}v^{q}(Q_{j}^{*})
        = \int_{Q_{j}^{*}} u v^{1-q} v^{q} \\
        & \leq \left(\sup_{Q_{j}^{*}} v^{q}\right)\, uv^{1-q}(Q_{j}^{*}) \\
        & \leq C \frac{v^{q}(Q_{j}^{*})}{|Q_{j}^{*}|} \int_{Q_{j}^{*}} u(x) \frac{\|u \|_{\Phi_{\varepsilon}, Q_{j}^{*}, v^{1-q}}}{\|u \|_{\Phi_{\varepsilon}, Q_{j}^{*}, v^{1-q}}} v(x)^{1-q} \,dx \\
        & \leq C \|u \|_{\Phi_{\varepsilon}, Q_{j}^{*},v^{1-q}} \frac{v^{q}(Q_{j}^{*})}{|Q_{j}^{*}|} \int_{Q_{j}^{*}} \Phi_{\varepsilon}\left(   \frac{u}{\|u \|_{\Phi_{\varepsilon}, Q_{j}^{*},v^{1-q}}} \right) v(x)^{1-q} \,dx   .
    \end{split}
\end{equation*}
Let us observe that the integral in the last expression is bounded by $v^{1-q}(Q_{j}^{*})$ by the definition of $\|u \|_{\Phi_{\varepsilon}, Q_{j}^{*}, v^{1-q}}$. On the other hand, since $v \in RH_{\infty}$, both $v^{q}$ and $v^{1-q}$ are doubling. Therefore, recalling that $v^{1-q}$ is an $A_1$ weight we obtain that  
\begin{equation*}
    \begin{split}
    uv(Q_{j}^{*})& \leq C\frac{v^{q}(Q_{j}^{*})}{|Q_{j}^{*}|} v^{1-q}(Q_{j}^{*})\|u \|_{\Phi_{\varepsilon}, Q_{j}^{*},v^{1-q}} \\ & \leq C v^{q}(Q_j) \|u \|_{\Phi_{\varepsilon}, Q_{j}^{*},v^{1-q}}  \frac{v^{1-q}(Q_j)}{|Q_j|} \\
        & \leq C \|u \|_{\Phi_{\varepsilon}, Q_{j}^{*},v^{1-q}}   \inf_{Q_j} v^{1-q} \int_{Q_j} v(x)^{q} \,dx \\
        & \leq C \|u \|_{\Phi_{\varepsilon}, Q_{j}^{*},v^{1-q}} \int_{Q_j} v(x) \,dx  \\
        & \leq \|u \|_{\Phi_{\varepsilon}, Q_{j}^{*},v^{1-q}} \frac{C}{\lambda} \int_{Q_j} f(x) v(x) \,dx  \\
        & \leq  \frac{C}{\lambda} \int_{Q_j} f(x)  M_{\Phi_{\varepsilon}, v^{1-q}}u (x) \,v(x) \,dx .
    \end{split}
\end{equation*}

We finally estimate $I_3$. Since each $h_j$ is supported in $Q_{j}$, we can use the integral representation of $T(h_{j}v)(x)$ for $x \notin Q_{j}^{*}$. By using Tchebyshev inequality, \eqref{FS prop h_j} and Tonelli theorem, we have that
\begin{equation*}
\begin{split}
     I_3 & \leq   uv\left( \left\{ x \in \mathbb{R}^{n} \backslash \Omega^{*} : \frac{|T(\sum_{j}h_{j}v)(x)|}{v(x)} > \frac{\lambda}{2} \right\} \right)  \\
     & \leq  uv\left( \left\{ x \in \mathbb{R}^{n} \backslash \Omega^{*} : \frac{|\sum_{j}T(h_{j}v)(x)|}{v(x)} > \frac{\lambda}{2} \right\} \right) \\
     & \leq \sum_{j} \frac{C}{\lambda} \int_{\mathbb{R}^n \backslash Q_{j}^{*}} \left|T(h_{j}v)(x)\right| u_{j}^{*}(x) \,dx \\
     & = \sum_{j} \frac{C}{\lambda} \int_{\mathbb{R}^n \backslash Q_{j}^{*}} \int_{Q_j} |K(x-y) - K(x-x_j)| |h_{j}(y)| v(y) \,dy \,\, u^{*}_{j}(x) \,dx \\
     & \leq \frac{C}{\lambda} \sum_{j}  \int_{Q_j} |h_{j}(y)| v(y)  \int_{\mathbb{R}^{n}\backslash Q_{j}^{*}}  |K(x-y) - K(x-x_j)| u^{*}_{j}(x) \,dx \,dy .
\end{split}
\end{equation*}

Let us consider the sets 
$A_{j,k}=\{ x : \sqrt{n}\, \ell_j \, 2^{k} < |x-x_{j}| \leq \sqrt{n}\, \ell_j \, 2^{k+1} \}$. Since $B(x_{j},2 \sqrt{n}\,\ell_j) \subset  Q_{j}^{*}$, by condition \eqref{OCZ: condicion de suavidad} we obtain that 
 \begin{equation*}
     \begin{split}
         I_{3} & \leq \frac{C}{\lambda} \sum_{j}  \int_{Q_j} |y-x_{j}| |h_{j}(y)| v(y) \sum_{k=1}^{\infty}\int_{A_{j,k}} \frac{|y-x_j|}{|x-x_{j}|^{n+1}} u^{*}_{j}(x) \,dx \,dy.
     \end{split}
 \end{equation*}
 If $y \in Q_j$, it follows that
  \begin{equation*}
     \begin{split}
         \sum_{k=1}^{\infty} \int_{A_{j,k}} \frac{|y-x_j|}{|x-x_{j}|^{n+1}} u^{*}_{j}(x) \,dx & \leq \sum_{k=1}^{\infty} \frac{\ell_j}{2} \frac{1}{(\sqrt{n}\, \ell_j \, 2^{k})^{n+1}} \int_{B(x_{j}, 2^{k+1}\sqrt{n}\, \ell_j)} u^{*}_{j}(x) \,dx \\
         & \leq  C \sum_{k=1}^{\infty} \frac{2^{-k}}{(\sqrt{n}\,\ell_j \, 2^{k+1})^{n}}\int_{B(x_{j}, 2^{k+1}\sqrt{n}\, \ell_j)} u^{*}_{j}(x) \,dx   \\
         & \leq C \inf_{Q}Mu^{*}_{j} \,\,\sum_{k=1}^{\infty} 2^{-k}  \\
         & \leq C \inf_{Q}Mu^{*}_{j}.
     \end{split}
 \end{equation*}
Thus, we get that 
 \begin{equation*}
       I_{3} \leq \frac{C}{\lambda} \sum_{j} \int_{Q_j} |h_{j}(y)| v(y) \inf_{Q_j} Mu^{*}_{j}(y) \,dy
 \end{equation*}
and from the definition of $h_j$, Lemma \ref{FS:Lema 2.5 Ana} and Lemma \ref{lema comparacion puntual maximal phi con peso} it follows that  
\begin{equation*}
    \begin{split}
        I_{3} & \leq \frac{C}{\lambda}  \sum_{j} \left( \int_{Q_j} f(x)Mu(x)\, v(x) \,dx ^{*} + f_{Q_{j}}^{v} \int_{Q_j} v(x) \,dx \, \inf_{Q_j}Mu^{*} \right) \\
        & \leq \frac{C}{\lambda}  \sum_{j} \int_{Q_{j}}f(x)M_{\Phi_{\varepsilon}, v^{1-q}}u^{*}(x) v(x) \,dx \\
        & \leq \frac{C}{\lambda} \int_{\mathbb{R}^{n}} f(x) M_{\Phi_{\varepsilon}, v^{1-q}}u(x) \, v(x) \,dx,
    \end{split}
\end{equation*} which completes the proof.
\end{proof}

We finally proof Theorem \ref{FS: FS para el conmutador}. 
\begin{proof}[Proof of Theorem \ref{FS: FS para el conmutador}]
We first consider the case $m=1$. Without loss of generality we can take $f$ non-negative with compact support and $\| b\|_{\text{BMO}}=1$. We consider the Calderón-Zygmund decomposition of $f$ at level $\lambda$ with respect to the measure $d\mu(x) =v(x)\,dx$ and set $f=g+h$  as in the proof of Theorem \ref{Teo FS}. We have that
\begin{equation*}
\begin{split}
    uv\left( \left\{ x \in \mathbb{R}^n : \frac{|T_{b}(fv)(x)|}{v(x)} > \lambda \right\} \right) \leq & uv\left( \left\{ x \in \mathbb{R}^n \backslash \Omega^*: \frac{|T_{b}(gv)(x)|}{v(x)} > \lambda /2 \right\} \right) + uv(\Omega^{*}) \\
    & + uv\left( \left\{ x \in \mathbb{R}^n \backslash{\Omega^{*}}: \frac{|T_{b}(hv)(x)|}{v(x)} > \lambda/2 \right\} \right)\\
    & = I^{1}_{1}+I^{1}_{2}+I^{1}_{3},
\end{split}
\end{equation*} where $\Omega$ and $\Omega^*$ are defined as in Theorem \ref{Teo FS}.

We start by estimating $I^{1}_1$. Let $\varepsilon>0$, $1 < p < \min \{q, 1+ \varepsilon /2 \}$ and  $u^{*}=u \,\chi_{\mathbb{R}^{n}\backslash\Omega^{*}}$, by Tchebyshev inequality and Theorem  \ref{FS: teo aux p}, we get that
\begin{equation*}
    \begin{split}
        I^{1}_{1} & \leq \frac{C}{\lambda^{p}} \int_{\mathbb{R}^n} |T_{b}(gv)(x)|^{p} u^{*}(x)v(x)^{1-p} \,dx \\
& \leq \frac{C}{\lambda^p} \int_{\mathbb{R}^n} |g(x)|^p M_{\Phi_{1+\varepsilon},v^{1-q}}u^{*}(x) v(x) \,dx  \\
& \leq \frac{C}{\lambda} \int_{\mathbb{R}^n} |g(x)| M_{\Phi_{1+\varepsilon},v^{1-q}}u^{*}(x) v(x) \,dx.
    \end{split}
\end{equation*} By the definition of $g$, we obtain that
\begin{equation*}
    \begin{split}
    I_{1}^{1} &\leq \frac{C}{\lambda} \left( \int_{\mathbb{R}^n} f(x) M_{\Phi_{1+\varepsilon}, v^{1-q}}u^{*}(x) v(x) \,dx + \sum_{j}  f^{v}_{Q_j} \int_{Q_j} M_{\Phi_{1+\varepsilon}, v^{1-q}}u^{*}(x) v(x) \,dx \right).  
    \end{split}
\end{equation*}
Therefore, for each $j$ we get that
\begin{equation*}
 \begin{split}
         f^{v}_{Q_j} \int_{Q_j} M_{\Phi_{1+\varepsilon}, v^{1-q}}u^{*}(x) v(x) \,dx &  \leq C f^{v}_{Q_j} v(Q_j) \inf_{Q_j}  M_{\Phi_{1+\varepsilon}, v^{1-q}}u^{*} \\
         & \leq \int_{Q_{j}} f(x) M_{\Phi_{1+\varepsilon}, v^{1-q}}u^{*}(x) v(x) \,dx,
 \end{split}
\end{equation*} where we have used Lemma \ref{FS:Lema 2.5 Ana}. So we have that
\begin{equation*}
    \begin{split}
       I_{1}^{1}& \leq \frac{C}{\lambda}  \left(\int_{\mathbb{R}^{n}} f(x) M_{\Phi_{1+\varepsilon}, v^{1-q}}u^{*}(x) v(x) \,dx    + \sum_{j}   \int_{Q_{j}} f(x) M_{\Phi_{1+\varepsilon}, v^{1-q}}u^{*}(x) v(x) \,dx \right) \\
        & \leq C \int_{\mathbb{R}^{n}} \frac{f(x)}{\lambda} M_{\Phi_{1+\varepsilon}, v^{1-q}}u(x) v(x) \,dx.
    \end{split}
\end{equation*}

In order to estimate $I_{2}^{1}$, we proceed as in the estimation of $I_2$ on Theorem \ref{Teo FS}. Taking into account that $t \leq \Phi_{1 + \varepsilon}(t)$, for $t\geq 0$, we get that 
\begin{equation*}
    \begin{split}
        uv(Q_{j}^{*}) 
        &  \leq \left(\sup_{Q_{j}^{*}} v^{q}\right)\, uv^{1-q}(Q_{j}^{*})\leq C \|u \|_{\Phi_{1+\varepsilon}, Q_{j}^{*},v^{1-q}} \frac{v^{q}(Q_{j}^{*})}{|Q_{j}^{*}|} \int_{Q_{j}^{*}} \Phi_{1+\varepsilon}\left(   \frac{u(x)}{\|u \|_{\Phi_{1+\varepsilon}, Q_{j}^{*},v^{1-q}}} \right) v(x)^{1-q} \,dx   .
    \end{split}
\end{equation*}
Let us note that the integral above is bounded by $v^{1-q}(Q_{j}^{*})$. Then, it follows that 
\begin{equation*}
    \begin{split}
   I_{2}^{1}&  \leq C v^{q}(Q_j) \,\inf_{Q_j} M_{\Phi_{1+\varepsilon}, v^{1-q}}u \, \, \frac{v^{1-q}(Q_j)}{|Q_j|} \leq C \, \inf_{Q_j} M_{\Phi_{1+\varepsilon}, v^{1-q}}u \,\, \inf_{Q_j} v^{1-q} \int_{Q_j} v(x)^{q}\,dx \\
        & \leq C \int_{Q_j} v(x) \,dx \inf_{Q_j} M_{\Phi_{1+\varepsilon}, v^{1-q}}u  \leq \frac{C}{\lambda} \int_{Q_j} f(x) v(x) \,dx \inf_{Q_j} M_{\Phi_{1+\varepsilon}, v^{1-q}}u \\
        & \leq  C \int_{Q_j} \frac{f(x)}{\lambda}  M_{\Phi_{1+\varepsilon}, v^{1-q}}u (x) v(x) \,dx .
    \end{split}
\end{equation*}

Finally, we estimate $I_{3}^{1}$. Recall that
\begin{equation*}
    T_{b}(hv)(x)= \sum_{j} (b-b_{Q_j}) T(h_{j}v)(x)- \sum_{j} T((b-b_{Q_j})h_{j}v)(x).
\end{equation*}
Therefore we get that
\begin{equation*}
    \begin{split}
        I_{3}^{1} \leq & uv\left( \left\{ x \in \mathbb{R}^n\backslash \Omega^{*} : \left|\sum_{j} \frac{(b-b_{Q_j}) T(h_{j}v)(x)}{v(x)}\right| > \frac{\lambda}{4}
  \right\} \right) \\ & \hspace{10mm}+  uv\left( \left\{ x \in \mathbb{R}^n\backslash \Omega^{*} : \left|\sum_{j} \frac{T((b-b_{Q_j})h_{j}v)(x)}{v(x)} \right| > \frac{\lambda}{4}   \right\} \right) \\
        & = I_{3,1}^{1} + I_{3,2}^{1}.
    \end{split}
\end{equation*}

 In order to estimate $I_{3,1}^{1}$, we use Tchebyshev inequality to obtain that 
\begin{equation*}
    \begin{split}
        I_{3,1}^{1} &\leq \frac{C}{\lambda} \sum_{j}\int_{\mathbb{R}^n \backslash \Omega^{*}}|b(x)-b_{Q_j}| T(h_{j}v)(x)u_{j}^{*}(x) \,dx \\
        & \leq \frac{C}{\lambda} \sum_{j}\int_{\mathbb{R}^{n} \backslash Q_{j}^{*}} |b(x)-b_{Q_j}| \left| \int_{Q_j} h_{j}(y)v(y) (K(x-y)-K(x-x_j)) \, dy \right| u(x) \,dx\\
        & \leq \frac{C}{\lambda} \sum_{j} \int_{Q_j} |h_{j}(y)| v(y) \int_{\mathbb{R}^{n} \backslash Q_{j}^{*}} |b(x)-b_{Q_j}| |K(x-y)-K(x-x_j) | u^{*}_{j}(x) \,dx \, dy.
    \end{split}
\end{equation*}

Consider the sets $A_{j,k}=\{ x : \sqrt{n}\, \ell_j \, 2^{k} < |x-x_{j}| \leq \sqrt{n}\, \ell_j \, 2^{k+1} \}$, by condition \eqref{OCZ: condicion de suavidad}, we get that
\begin{equation*}
    \begin{split}
        \int_{\mathbb{R}^{n} \backslash Q_{j}^{*}} |b(x)-b_{Q_j}| & |K(x-y)-K(x-x_j) | u^{*}_{j}(x) \,dx \\ &= \sum_{k=1}^{\infty} \int_{A_{j,k}} |b(x)-b_{Q_j}|  |K(x-y)-K(x-x_j) | u^{*}_{j}(x) \,dx \\& \leq \sum_{k=1}^{\infty}\int_{A_{j,k}}  |b(x)-b_{Q_j}| \frac{|y-x_j|}{|x-x_j|^{n+1}} u^{*}_{j}(x) \,dx \\
        &  \leq C  \sum_{k=1}^{\infty} \frac{\ell_j}{\sqrt{n}\ell_j 2^k} \frac{1}{(\sqrt{n}\ell_j 2^{k+1})^n} \int_{B(x_j, \sqrt{n}\ell_j 2^{k+1})} |b(x)-b_{Q_j}| u^{*}_{j}(x) \,dx  \\
        &\leq C \sum_{k=1}^{\infty}  \frac{2^{-k}}{|\sqrt{n} 2^{k+2} Q_j |} \int_{\sqrt{n} 2^{k+2} Q_j } |b(x)-b_{Q_j}| u^{*}_{j}(x) \,dx.
    \end{split}
\end{equation*}

Let $k_0 \in \mathbb{Z}$ such that $2^{k_0 -1} \leq \sqrt{n}< 2^{k_0}$. Observe that $\tilde{\Phi}_{1+\varepsilon}(t)= (e^{t^{1/(1+\varepsilon)}}-1) \chi_{(1,\infty)}(t)$. By applying Proposition \ref{BMO valor abs con promedios en 2k Q}, the generalized Hölder inequality (Theorem \ref{Holder gen}) and Proposition \ref{BMO comparacion norma phi y bmo} we get that 
\begin{equation*}
    \begin{split}
    \frac{1}{|\sqrt{n}  2^{k+2} Q_j |} \int_{\sqrt{n}  2^{k+2} Q_j } |b(x)-b_{Q_j}| u^{*}_{j}(x) \,dx & \leq     
    \frac{C}{|2^{k+k_0 +2} Q_j|} \int_{ 2^{k+k_0 +2} Q_j} |b(x)-b_{ 2^{k+k_0 +2} Q_j}| u^{*}_{j}(x) \,dx \\ & \hspace{7mm}+ C (k+ k_0 +2) M u^{*}_{j}(y) \\& \leq  C \| b-b_{ 2^{k+k_0 +2} Q_j} \|_{\tilde{\Phi}_{1+\varepsilon}, \sqrt{n}  2^{k+ k_0 +2} Q_j} \|u^{*}_{j} \|_{\Phi_{1+\varepsilon}, \sqrt{n}  2^{k+ k_0 +2} Q_j} \\
    & \hspace{7mm} + C (k+ k_0 +2) M u^{*}_{j}(y) \\ & \leq C k \,M_{\Phi_{1+\varepsilon}} u^{*}_{j}(y). 
    \end{split}
\end{equation*} Therefore, 
\begin{align*}
\int_{\mathbb{R}^{n} \backslash Q_{j}^{*}} |b(x)-b_{Q_j}|  |K(x-y)-K(x-x_j) | u^{*}_{j}(x) \,dx \leq C \sum_{k=1}^{\infty}  2^{-k} k M_{\Phi_{1+\varepsilon}}u_{j}^{*}(y) \leq C M_{\Phi_{1+\varepsilon}}u_{j}^{*}(y),
\end{align*} for almost every $y \in Q_j$. Then we have that
\begin{equation*}
    \begin{split}
        I_{3,1}^{1}& \leq \frac{C}{\lambda} \sum_{j} \int_{Q_j} |h_{j}(y)| M_{\Phi_{1+\varepsilon}} u^{*}_{j}(y) v(y) \,dy \\ & \leq \frac{C}{\lambda} 
\sum_j \inf_{Q_j} M_{\Phi_{1+\varepsilon}}u^{*}_{j}  \int_{Q_j} fv + C \sum_{j} \inf_{Q_j} M_{\Phi_{1+\varepsilon}}u^{*}_{j}  \int_{Q_j} f^{v}_{Q_j}v \\
& \leq \frac{C}{\lambda} \int_{\mathbb{R}^n} f(x) M_{\Phi_{1+\varepsilon}}u(x) v(x)\,dx + \frac{C}{\lambda} \sum_{j}  \inf_{Q_j} M_{\Phi_{1+\varepsilon}}u^{*}_{j} \int_{Q_j} fv \\ & \leq \frac{C}{\lambda} \int_{\mathbb{R}^n} f(x) M_{\Phi_{1+\varepsilon, v^{1-q}}}u(x) v(x)\,dx.
    \end{split}
\end{equation*}

 Let us now estimate $I_{3,2}^{1}$. Take $\delta = 1+ \varepsilon$, by Theorem \ref{Teo FS} we obtain that 
\begin{align*}
I_{3,2}^{1} &\leq  u^{*}v\left( \left\{ x \in \mathbb{R}^n : \left| \frac{ T \left(\sum_{j} (b-b_{Q_j})h_{j}v\right)(x)}{v(x)} \right| > \frac{\lambda}{4}  \right\} \right) \\
&\leq  \frac{C}{\lambda}\int_{\mathbb{R}^n} \left| \sum_{j} (b(x)-b_{Q_j}) h_{j}(x) \right| M_{\Phi_{1+\varepsilon}, v^{1-q}}u^{*}(x) v (x) \, dx\\
& \leq  \frac{C}{\lambda} \sum_{j} \int_{Q_j} |b(x)-b_{Q_j}| |h_{j}(x)| M_{\Phi_{1+\varepsilon}, v^{1-q}}u^{*}(x) v (x) \, dx\\ & \leq \frac{C}{\lambda}  \sum_{j}\inf_{Q_j}M_{\Phi_{1+\varepsilon}, v^{1-q}}u^{*} \int_{Q_j} |b(x)-b_{Q_j}|  |h_{j}(x)| v (x) \, dx.
\end{align*}

Let us note that, since 
\begin{equation}\label{C}
    \|g \|_{\Phi, Q, v} \approx \inf_{\mu > 0} \left\{ \mu + \frac{\mu}{v(Q)} \int_{Q} \Phi \left( \frac{|g|}{\mu}\right) v\right\}
\end{equation} by the generalized Hölder inequality we obtain
\begin{align*}
    \int_{Q_j} |b(x) - b_{Q_j}| f(x) v(x) \, dx & \leq C v(Q_j)  \|b-b_{Q_j} \|_{\tilde{\Phi}_{1}, Q_{j},v} \|f \|_{\Phi_{1}, Q_j,v}  \leq C v(Q_j) \|f \|_{\Phi_{1}, Q_j,v} \\ & \leq  C v(Q_j) \left( \lambda +\frac{\lambda}{v(Q_j)} \int_{Q_j} \Phi_{1} \left(\frac{|f(x)|}{\lambda} \right) v(x)\,dx     \right)\\
        & \leq C \left( \lambda v(Q_j)+ \lambda \int_{Q_j} \Phi_{1} \left( \frac{|f(x)|}{\lambda} \right) v(x)\,dx \right) \\
        & \leq C \left( \int_{Q_j} fv + \lambda \int_{Q_j} \Phi_{1} \left(\frac{|f(x)|}{\lambda} \right) v(x)\,dx \right) .
\end{align*}
On the other hand, 
\begin{equation*}
    \begin{split}
      \int_{Q_j} |b(x)-b_{Q_j}| | h_{j}(x)| v (x) \, dx & \leq  \int_{Q_j} |b(x) - b_{Q_j}| f(x) v(x) \, dx + f_{Q_{j}}^{v} \int_{Q_j} |b(x) - b_{Q_j}| v(x) \,dx  ,
    \end{split}
\end{equation*}
and since $v \in RH_{\infty}$, we get that
\begin{equation*}
    \begin{split}
        f_{Q_{j}}^{v} \int_{Q_j} |b(x) - b_{Q_j}| v(x) \,dx &\leq \frac{v(Q_j)}{|Q_j|} \frac{1}{v(Q_j)} \int_{Q_j} fv \int_{Q_j} |b-b_{Q_j}| \\
        & \leq \int_{Q_j}fv .
    \end{split}
\end{equation*}

These last estimates yields 
\begin{equation*}
    \begin{split}
     I_{3,2}^{1} & \leq \frac{C}{\lambda}  \sum_{j}\inf_{Q_j}M_{\Phi_{1+\varepsilon}, v^{1-q}}u^{*} \int_{Q_j} |b(x)-b_{Q_j}||h_j(x)|  v (x) \, dx \\
&  \leq \frac{C}{\lambda}  \sum_{j}\inf_{Q_j}M_{\Phi_{1+\varepsilon}, v^{1-q}}u^{*} \left(\int_{Q_j} |b(x) - b_{Q_j}| f(x) v(x) \, dx + f_{Q_{j}}^{v} \int_{Q_j} |b(x) - b_{Q_j}| v(x) \,dx   \right)
     \\
        &\leq \frac{C}{\lambda}  \sum_{j}\inf_{Q_j}M_{\Phi_{1+\varepsilon}, v^{1-q}}u^{*} \left(  \int_{Q_j} f(x)v(x)\,dx + \lambda \int_{Q_j} \Phi_{1} \left(\frac{|f(x)|}{\lambda} \right) v(x)\,dx     \right) \\ 
        & \leq C \sum_{j} \int_{Q_j} \frac{f(x)}{\lambda}M_{\Phi_{1+\varepsilon}, v^{1-q}}u^{*}(x)v(x)  \,dx + C\sum_{j} \int_{Q_{j} }\Phi_{1} \left(\frac{|f(x)|}{\lambda} \right)M_{\Phi_{1+\varepsilon}, v^{1-q}}u^{*} (x) v(x)  \,dx.
    \end{split}
\end{equation*}

From the fact that $t \leq \Phi_{1+\varepsilon}$, we obtain that $I_{3,2}^{1}$ is bounded by
\begin{equation*}
   I_{3,2}^{1} \leq C\int_{\mathbb{R}^n} \Phi_{1} \left(\frac{|f(x)|}{\lambda} \right)   M_{\Phi_{1+\varepsilon}, v^{1-q}}u (x) v(x) \,dx,
\end{equation*} which completes the proof for the case $m=1$. 

Let us now consider the case $m>1$. We shall prove this by induction. Fix $m >1$ and suppose the hypothesis is true for $T_{b}^{k}$ for every $1 \leq k \leq m-1$. We consider the same Calderón-Zygmund decomposition as in the case $m=1$, obtaining that
\begin{equation*}
    \begin{split}
            uv\left( \left\{ x \in \mathbb{R}^n : \frac{|T^{m}_{b}(fv)(x)|}{v(x)} > \lambda \right\} \right) &\leq  uv\left( \left\{ x \in \mathbb{R}^n \backslash \Omega^*: \frac{|T^{m}_{b}(gv)(x)|}{v(x)} > \frac{\lambda}{2} \right\} \right) + uv(\Omega^{*}) \\
    & \quad + uv\left( \left\{ x \in \mathbb{R}^n \backslash{\Omega^{*}}: \frac{|T^{m}_{b}(hv)(x)|}{v(x)} > \frac{\lambda}{2}\right\} \right)\\
    & = I^{m}_{1}+I^{m}_{2}+I^{m}_{3}.
    \end{split}
\end{equation*}

To estimate $I^{m}_{1}$ we proceed as in the case $m=1$, applying Theorem \ref{FS: teo aux p},  concluding that
\begin{equation*}
       I_{1}^{m} \leq \frac{C}{\lambda^p} \int_{\mathbb{R}^n} |T^{m}_{b}(gv)(x)|^{p} u^{*}(x) v(x)^{1-p} \leq \frac{C}{\lambda}\int_{\mathbb{R}^n} |g(x)| M_{\Phi_{m+\varepsilon}, v^{1-q}}u^{*}(x)v(x) \,dx,
\end{equation*}
which implies
\begin{equation*}
    I_{1}^{m} \leq \frac{C}{\lambda} \int_{\mathbb{R}^n} f(x) M_{\Phi_{m+\varepsilon}, v^{1-q}}u(x) \,\, v(x)^{1-p} \,dx.
\end{equation*} The term  $I_{2}^{m}$ can also be bounded in the same way we did for $m=1$. 

 At last, we estimate the term $I_{3}^{m}$. We have that 
\begin{equation*}
    T^{m}_{b}(hv)= \sum_{j}(b-b_{Q_j})^{m}T(h_{j}v) - \sum_{j} T((b - b_{Q_j})^{m} h_{j}v ) - \sum_{j} \sum_{i=1}^{m-1} C_{m,i} T_{b}^{i} ((b - b_{Q_j})^{m-i} h_{j}v ) ,
\end{equation*} and therefore
\begin{equation*}
    \begin{split}
        I_{3}^{m} &\leq uv\left( \left\{   x\in \mathbb{R}^{n} \backslash\Omega^{*} : \left|  \frac{\sum_{j}(b-b_{Q_j})^{m}T(h_{j}v)}{v}\right| >\frac{\lambda}{6}  \right\}  \right)  \\ & \hspace{7mm}+uv\left( \left\{   x\in \mathbb{R}^{n} \backslash\Omega^{*} : \left|  \frac{\sum_{j} T((b - b_{Q_j})^{m} h_{j}v )}{v}\right|  >\frac{\lambda}{6}\right\}  \right)   \\&  \hspace{14mm}+uv\left( \left\{   x\in \mathbb{R}^{n} \backslash\Omega^{*} : \left|  \frac{\sum_{j} \sum_{i=1}^{m-1} C_{m,i} T_{b}^{i} ((b - b_{Q_j})^{m-i} h_{j}v )}{v}\right| >\frac{\lambda}{6C_0}\right\}  \right)  \\
        & =I^{m}_{3,1} +  I^{m}_{3,2}  + I^{m}_{3,3} ,
    \end{split}
\end{equation*}
where $\displaystyle C_0=\max_{1 \leq i \leq m-1}C_{m,i} $. 

We will now estimate each one of those terms. For $ I^{m}_{3,1}$, by Tchebyshev inequality and condition \eqref{OCZ: condicion de suavidad} we get that
\begin{equation*}
    \begin{split}
         I^{m}_{3,1} & \leq \frac{6}{\lambda} \int_{\mathbb{R}^{n}\backslash \Omega^{*}} \left| \sum_{j} (b(x) - b_{Q_j})^{m} T(h_j v)(x)   \right| u^{*}(x) \,dx \\
         & \leq \frac{C}{\lambda} \sum_j \int_{Q_j} |h_j(y)| v(y) \int_{\mathbb{R}^{n}\backslash Q_{j}^{*}} |b(x) - b_{Q_j}|^{m} |K(x-y) - K(x-x_{j})| u^{*}_{j}(x) \, dx dy\\
&  \leq \frac{C}{\lambda} \sum_j \int_{Q_j} |h_j(y)| v(y) \sum_{k=1}^{\infty}
 \int_{A_{j,k}} |b(x) - b_{Q_j}|^{m} \frac{|y-x_{Q_j} |}{|x-x_{j} |^{n+1}} u^{*}_{j}(x) \, dx dy   
    \end{split}
\end{equation*} where the sets $A_{j,k}$ are the ones defined in the case $m=1$. 
Taking $k_0$ as before, the inner sum can be bounded by   
\begin{equation*}
    \begin{split}
\sum_{k=1}^{\infty} \int_{A_{j,k}} |b(x) - b_{Q_j}|^{m} \frac{|y-x_j |}{|x-x_j |^{n+1}} u^{*}_{j}(x) \, dx &  \leq C\sum_{k=1}^{\infty} \frac{\ell_j}{\sqrt{n}\, \ell_j \, 2^{k}} \frac{1}{(\sqrt{n}\, \ell_j \, 2^{k})^n} \int_{B(x_j , 2^{k+1}\sqrt{n}\, \ell_j \, )} |b - b_{Q_j}|^{m}  u^{*}_{j}   \\
        & \leq C \sum_{k=1}^{\infty} \frac{2^{-k}}{|2^{k+k_{0}+2}Q_j|}\int_{2^{k+k_{0}+2}Q_j}  |b - b_{Q_j}|^{m}  u^{*}_{j} .
    \end{split}
\end{equation*} Since
\begin{equation}\label{B}
    \| g_{0}^{m}\|_{\tilde{\Phi}_{m}, Q} \approx \|g_{0} \|_{\tilde{\Phi}_{1},Q}^{m},
\end{equation}
   we apply Proposition \ref{BMO valor abs con promedios en 2k Q} together with the generalized Hölder inequality (Theorem \ref{Holder gen}) and Proposition \ref{BMO comparacion norma phi y bmo}, in order to get
\begin{equation*}
    \begin{split}
     \frac{1}{|2^{k+k_{0}+2}Q_j|}\int_{2^{k+k_{0}+2}Q_j}  |b(x) - b_{Q_j}|^{m}  u^{*}_{j}(x) \, dx &  \leq \frac{C}{|2^{k+k_{0}+2}Q_j|} \int_{2^{k+k_{0}+2}Q_j}  |b(x) - b_{2^{k+k_{0}+2} Q_j}|^{m}  u^{*}_{j}(x) \, dx \\ & \hspace{5mm}+ C (k+k_{0}+2)^{m} Mu^{*}_{j}(y) \\
     & \leq C \| |b-b_{2^{k+k_{0}+2} Q_j}|^{m}\|_{\tilde{\Phi}_{m}, 2^{k+k_{0}+2}Q_j} \| u^{*}_{j}\|_{\Phi_{m},2^{k+k_{0}+2}Q_j } \\
     & \hspace{5mm}+ C (k+k_{0}+2)^{m} Mu^{*}_{j}(y) \\
     & \leq C\| |b-b_{2^{k+k_{0}+2} Q_j}|\|^{m}_{\tilde{\Phi}_{1}, 2^{k+k_{0}+2}Q_j} \| u^{*}_{j}\|_{\Phi_{m},2^{k+k_{0}+2}Q_j } \\
     & \hspace{5mm}+ C (k+k_{0}+2)^{m} Mu^{*}_{j}(y) \\
     & \leq C k^{m} \, M_{\Phi_{m}} u^{*}_{j}(y),
    \end{split}
\end{equation*} for almost every $y \in Q_j$. Returning to $I_{3,1}^{m}$, we have that
\begin{align*}
    I_{3,1}^{m} &\leq \frac{C}{\lambda} \sum_j \int_{Q_j} |h_j(y)| v(y) M_{\Phi_{m}}u_{j}^{*}(y)\sum_{k=1}^{\infty} 2^{-k}k^{m} \,dy \\&\leq  \frac{C}{\lambda} \sum_j \int_{Q_j} |h_j(y)| v(y) M_{\Phi_{m}}u_{j}^{*}(y)\,dy \\
    & \leq \frac{C}{\lambda} \sum_{j} \int_{Q_j} |h_{j}(y)| v(y) M_{\Phi_{m+\varepsilon}} u^{*}_{j}(y) \, dy,
\end{align*} where we have used that $M_{\Phi_{m}} u^{*}_{j} \leq M_{\Phi_{m+\varepsilon}} u^{*}_{j} $. From this estimate, we can obtain the desired bound by proceeding similarly as in $I_{3,1}^{1}$.

We now estimate $I_{3,2}^{m}$. By Theorem  \ref{Teo FS} with $m+ \varepsilon$, 
\begin{equation*}
\begin{split}
     I_{3,2}^{m} &\leq  u^{*}v\left( \left\{ x \in \mathbb{R}^n : \left| \frac{ T \left(\sum_{j} (b-b_{Q_j})^{m}h_{j}v\right)(x)}{v(x)} \right| > \frac{\lambda}{4}   \right\} \right) \\
     & \leq \frac{C}{\lambda} \int_{\mathbb{R}^n} \left| \sum_{j} (b(x)-b_{Q_j})^{m} h_{j}(x) \right| M_{\Phi_{m+\varepsilon}, v^{1-q}}u^{*}(x) v (x) \, dx\\
& \leq  \frac{C}{\lambda} \sum_{j} \inf_{Q_j}M_{\Phi_{m+\varepsilon}, v^{1-q}}u^{*} \, \int_{Q_j} |b(x)-b_{Q_j}|^{m} f(x) v (x) \, dx\\
& \hspace{5mm}+  \frac{C}{\lambda} \sum_{j} \inf_{Q_j}M_{\Phi_{m+\varepsilon}, v^{1-q}}u^{*} \, f_{Q_j}^{v} \,\int_{Q_j} |b(x)-b_{Q_j}|^{m}  v (x) \, dx.
\end{split}
\end{equation*} Combining the generalized Hölder inequality (Theorem \ref{Holder gen}) with \eqref{B}, Proposition \ref{BMO comparacion norma phi y bmo} and \eqref{C} we arrive to 
\begin{equation*}
    \begin{split}
       \int_{Q_j} |b(x) - b_{Q_j}|^{m} f(x) v(x) \, dx & \leq C v(Q_j)  \|(b-b_{Q_j})^{m} \|_{\tilde{\Phi}_{m}, Q_{j},v} \|f \|_{\Phi_{m}, Q_j,v} \\
        & \leq C v(Q_j) \|f \|_{\Phi_{m}, Q_j,v} \\ & \leq C v(Q_j) \left( \lambda +\frac{\lambda}{v(Q_j)} \int_{Q_j} \Phi_{m} \left(\frac{|f(x)|}{\lambda} \right) v(x) \,dx    \right)\\
        & \leq  C \lambda v(Q_j)+ \lambda \int_{Q_j} \Phi_{m} \left( \frac{|f(x)|}{\lambda} \right) v(x) \,dx \\
        & \leq C \int_{Q_j} f(x)v(x)\,dx + \lambda \int_{Q_j} \Phi_{m} \left(\frac{|f(x)|}{\lambda} \right) v(x)\,dx .
    \end{split}
\end{equation*} Since $v \in RH_{\infty}$ and  $\displaystyle \left(\frac{1}{|Q_j|} \int_{Q_{j}} |b-b_{Q_j}|^{m} \right)^{1/m} \leq C \| b\|_{\text{BMO}} $, let us also observe that 
\begin{align*}
    \begin{split}
        f_{Q_{j}}^{v} \int_{Q_j} |b(x) - b_{Q_j}|^{m} v(x) \,dx &\leq C \frac{v(Q_j)}{|Q_j|} \frac{1}{v(Q_j)} \int_{Q_j} f(x)v(x)\,dx 
 \, \int_{Q_j} |b(x)-b_{Q_j}|^{m} \,dx \\
        & \leq C \int_{Q_j}f(x)v(x) \,dx .
    \end{split}
\end{align*}
Combining these last two estimates, for $I_{3,2}^{m}$ we have that
\begin{equation*}
    \begin{split}
            I_{3,2}^{m} 
        &\leq C \sum_{j}\inf_{Q_j}M_{\Phi_{m+\varepsilon}, v^{1-q}}u^{*} \left(   \int_{Q_j}\frac{f(x)}{\lambda}v(x)\,dx  +  \int_{Q_j} \Phi_{m} \left(\frac{|f(x)|}{\lambda} \right) v (x) \,dx    \right) \\ 
        & \leq C\sum_{j} \int_{Q_{j} }\Phi_{m} \left(\frac{|f(x)|}{\lambda} \right)   M_{\Phi_{m+\varepsilon}, v^{1-q}}u^{*}(x) v(x) \,dx \\
        & \leq C \int_{\mathbb{R}^n} \Phi_{m} \left(\frac{|f(x)|}{\lambda} \right)   M_{\Phi_{m+\varepsilon}, v^{1-q}}u (x) v(x) \,dx.
    \end{split}
\end{equation*}

It only remains to estimate $I_{3,3}^{m}$. Indeed, by applying the inductive hypothesis we get that
\begin{equation*}
    \begin{split}
        I_{3,3}^{m} & \leq \sum_{i=1}^{m-1}  u^{*} v \left( \left\{ \frac{T^{i}_{b}(\sum_j (b-b_{Q_j})^{m-i} h_{j} v )(x)}{v(x)} > \frac{\lambda}{6C}   \right\}  \right) \\
        & \leq C \sum_{i=1}^{m-1} \int_{\mathbb{R}^n} \Phi_{i} \left(\frac{\sum_j |b(x)-b_{Q_j}|^{m-i} |h_{j}(x)|}{\lambda}\right) M_{\Phi_{i+\varepsilon}, v^{1-q'}}u^{*}(x) v(x) \,dx  \\
        & \leq C \sum_{i=1}^{m} \sum_{j} \int_{Q_j} \Phi_{i} \left(\frac{\sum_j |b(x)-b_{Q_j}|^{m-i} |h_{j}(x)|}{\lambda}\right) M_{\Phi_{i+\varepsilon}, v^{1-q'}}u_{j}^{*}(x) v(x) \,dx  \\
        &\leq C \sum_{i=1}^{m} \sum_{j}  \left( \inf_{Q_j} M_{\Phi_{i+\varepsilon}, v^{1-q' }} u^{*}_{j} \right) \int_{Q_j} \Phi_{i} \left(\frac{\sum_j |b(x)-b_{Q_j}|^{m-i} |h_{j}(x)|}{\lambda} \right)v(x) \,dx.
    \end{split}
\end{equation*}
Since $\displaystyle\Phi_{i}^{-1}(t)\approx \frac{t}{(1+ \log^{+} t)^{i}}$ for $1\leq i \leq m$, if $\displaystyle \Psi (t) \approx (e^{ t^{1/i} }-1)\chi_{(1,\infty)}(t) $, it follows that
\begin{equation*}
    \Phi_{m}^{-1}(t) \, \Psi_{m-i}^{-1} (t) \approx \frac{t}{(1+ \log^{+} t)^{m}} \log(1+ t)^{m-i} \approx \frac{t}{(1+ \log^{+} t)^{i}}\approx \Phi_{i}^{-1} (t).
\end{equation*}
By the above, we have that
\begin{equation*}\begin{split}
    \int_{Q_j}  \Phi_{i} \left(\frac{\sum_j |b-b_{Q_j}|^{m-i} |h_{j}|}{\lambda} \right)v  \leq  C \left( \int_{Q_j} \Phi_{m}\left( \frac{|h_j |}{\lambda} \right) v
   +
    \int_{Q_j} \Psi_{m-i} (|b - b_{Q_j}|^{m-i}) v \right).
    \end{split}
\end{equation*}
Since $\Phi_{m}$ is convex, for the first integral we get that
\begin{equation*}\begin{split}
    \int_{Q_j} \Phi_{m} \left( \frac{|h_j (x)|}{\lambda} \right) v(x)\,dx & \leq C \left( \int_{Q_j} \Phi_{m}\left(\frac{f(x)}{\lambda} \right) v(x) \,dx + \int_{Q_j} \Phi_{m} \left( \frac{f_{Q_j}^{v}}{\lambda} \right)v(x) \,dx\right) \\
    & \leq C \left( \int_{Q_j} \Phi_{m}\left(\frac{f(x)}{\lambda} \right) v(x) \,dx  + v(Q_j) \right)\\ & \leq C \int_{Q_j} \Phi_{m}\left(\frac{f(x)}{\lambda} \right) v(x) \,dx .
    \end{split}
\end{equation*} On the second integral, by Proposition \ref{BMO comparacion norma phi y bmo} we have that
\begin{equation*}
    \begin{split}
        \int_{Q_j} \Psi_{m-i} (|b(x) - b_{Q_j}|^{m-i}) v(x) \,dx & = \int_{Q_j} \Psi (|b(x)-b_{Q_j}|) v(x) \,dx \\
        & \leq C \int_{Q_j} \Psi \left(\frac{|b(x)-b_{Q_j}|}{\| b - b_{Q_j} \|_{\Psi, Q_{j},v}}\right) v(x) \,dx \\
        &  \leq v(Q_j) \leq C \int_{Q_j} \Phi_{m}\left(\frac{f(x)}{\lambda}  \right)v(x) \,dx.
    \end{split}
\end{equation*}
Therefore, we get the desired estimate in a similar way as done for $I_{3,3}^{1}$.
\end{proof}

\end{document}